\documentclass[12pt]{amsart}
\usepackage{fullpage,url,amssymb,enumerate,colonequals}

\usepackage{mathrsfs} 
\usepackage{MnSymbol}
\usepackage{extarrows}
\usepackage{lscape}
\usepackage[all,cmtip]{xy}

\usepackage[OT2,T1]{fontenc}

\usepackage{color}

\usepackage[
        colorlinks, citecolor=darkgreen,
        backref,
        pdfauthor={Alain Kraus}, 
]{hyperref}

\usepackage[francais]{babel}

\newcommand{\C}{\mathbb{C}}
\newcommand{\F}{\mathbb{F}}

\newcommand{\Q}{\mathbb{Q}}

\newcommand{\Z}{\mathbb{Z}}

\newcommand{\Kbar}{{\overline{K}}}

\newcommand{\sym}{\mathbb{S}}

\newcommand{\ff}{\mathfrak{f}}
\newcommand{\fq}{\mathfrak{q}}

\newcommand{\pp}{\mathfrak{p}}
\newcommand{\mm}{\mathfrak{m}}
\newcommand{\nn}{\mathfrak{n}}

\newcommand{\calA}{\mathcal{A}}

\newcommand{\calH}{\mathcal{H}}

\newcommand{\calS}{\mathcal{S}}

\newcommand{\LL}{\mathscr{L}}

\DeclareMathOperator{\Aut}{Aut}

\DeclareMathOperator{\Gal}{Gal}

\DeclareMathOperator{\Norm}{Norm}

\DeclareMathOperator{\Res}{Res}

\DeclareMathOperator{\pgcd}{pgcd}

\DeclareMathOperator{\Max}{{Max}}

\newcommand{\GalK}{{\Gal}(\Kbar/K)}

\newcommand{\GL}{\operatorname{GL}}

\numberwithin{equation}{section}

\newtheorem{theoreme}{Th\'eor\`eme}
\newtheorem{lemme}{Lemme}
\newtheorem{corollaire}{Corollaire}
\newtheorem{proposition}{Proposition}

\theoremstyle{definition}

\newtheorem{question}[equation]{Question}

\theoremstyle{remark}
\newtheorem{remark}[equation]{Remarque}

\definecolor{darkgreen}{rgb}{0,0.5,0}

\begin{document}

\title{Le th\'eor\`eme de Fermat sur certains corps  de nombres totalement r\'eels}

\author{Alain Kraus}

\address{Universit\'e Pierre et Marie Curie - Paris 6,
Institut de Math\'ematiques de Jussieu,
4 Place Jussieu, 75005 Paris, 
France}
\email{alain.kraus@imj-prg.fr}

\begin{abstract} Soit $K$ un corps de nombres totalement r\'eel. Pour tout nombre premier $p\geq 5$, notons $F_p$ la courbe de Fermat d'\'equation
$x^p+y^p+z^p=0$.  Sous l'hypoth\`ese que    $2$ est  totalement ramifi\'e dans $K$, on \'etablit quelques  r\'esultats  sur    l'ensemble $F_p(K)$ des points de $F_p$ rationnels sur  $K$.  On obtient  un crit\`ere pour que le th\'eor\`eme de Fermat asymptotique soit vrai  sur $K$, crit\`ere relatif \`a  l'ensemble  des newforms modulaires paraboliques de Hilbert sur $K$, de poids parall\`ele $2$ et de niveau l'id\'eal premier au-dessus de $2$. 
Il  peut souvent  se  tester   simplement   num\'eriquement, notamment   quand   le nombre de classes restreint de $K$ vaut $1$. 
Par ailleurs, en utilisant la m\'ethode modulaire, 
on d\'emontre  le th\'eor\`eme de Fermat de fa\c con effective, sur  certains corps de nombres dont les  degr\'es  sur $\Q$ sont $3,4,5,6$ et $8$.
\end{abstract}



\date{\today}

\maketitle

\renewcommand{\abstractname}{Abstract}
\begin{abstract}   Let $K$ be a totally real number field. For all prime number $p\geq 5$, let us denote by $F_p$ the Fermat curve of equation
$x^p+y^p+z^p=0$. Under the assumption that   $2$ is totally ramified in  $K$, we  establish some results about the set $F_p(K)$ of   points of $F_p$ rational over $K$. We obtain a criterion so  that  the asymptotic Fermat's Last Theorem is true over $K$, criterion related to 
 the set of  Hilbert modular cusp newforms over $K$, of parallel weight $2$ and of level the prime ideal above $2$. It   is often simply testable numerically,  particularly if 
 the narrow class number of $K$ is $1$. Furthermore, using the modular method, 
we prove  Fermat's Last Theorem effectively, over some number fields whose degrees over $\Q$ are   $3,4,5,6$ and $8$.
\end{abstract}

\bigskip\smallskip

\section{Introduction}

Soient $K$ un corps de nombres totalement r\'eel et  $p\geq 5$ un nombre premier. Notons
\begin{equation}
F_p : x^p+y^p+z^p=0
\label{(1.1)}
\end{equation}
la courbe de Fermat d'exposant $p$.  Dans le cas o\`u  $2$ est  totalement ramifi\'e dans $K$, on se propose   de faire quelques remarques sur  la description de  l'ensemble $F_p(K)$ des points de $F_p$  rationnels sur $K$. On  se pr\'eoccupera  en particulier  de cette description  quand de plus   le nombre de classes restreint  de $K$ vaut $1$.

Adoptons la terminologie selon laquelle $F_p(K)$ est trivial,  si pour tout point $(x,y,z)\in F_p(K)$ on a $xyz=0$. Tel est le cas si $K=\Q$ (\cite{Wiles}). On dira que le th\'eor\`eme de Fermat asymptotique est vrai sur $K$,  si $F_p(K)$ est trivial d\`es que $p$ est plus grand qu'une constante qui ne d\'epend que de $K$. Parce que $K$ ne contient pas les racines cubiques de l'unit\'e, la conjecture $abc$ sur $K$ implique le th\'eor\`eme de Fermat asymptotique sur $K$ (cf. \cite{Browkin}). 

Au cours de ces derni\`eres ann\'ees,   les principaux r\'esultats qui ont \'et\'e \'etablis concernant l'\'equation de Fermat sur les corps totalement r\' eels sont dus \`a Freitas et Siksek. Ils ont notamment obtenu  un crit\`ere permettant parfois de d\' emontrer le th\'eor\`eme de Fermat asymptotique sur un corps totalement r\'eel  (\cite{FS1}, th. 3). En particulier, ils en ont d\'eduit le th\'eor\`eme de Fermat asymptotique pour une proportion de $5/6$ de corps quadratiques r\'eels. Ils ont par ailleurs d\' emontr\'e que  pour   $K=\Q\left(\sqrt{m}\right)$, o\`u $m\leq 23$ est un entier  sans facteurs carr\'es,  autre que  $5$ et $17$, l'ensemble 
$F_p(K)$ est trivial pour tout   $p\geq 5$ (\cite{FS3}).  Le cas o\`u $m=2$ avait d\'ej\`a \'et\'e  \'etabli  dans \cite{JarvisMeekin}.

\subsection{Le crit\`ere de Freitas et Siksek} \'Enon\c cons leur r\'esultat dans le cas   o\`u $2$ est totalement ramifi\'e dans $K$.  

On notera dans toute la suite, 
$d$ le degr\'e de $K$ sur $\Q$, 
$O_K$ l'anneau d'entiers de $K$ et 
$\LL$ l'id\'eal premier de $O_K$ au-dessus de $2$. On a $2O_K=\LL^d$.

Soient $v_{\LL}$ la valuation sur  $K$ associ\'ee \`a $\LL$ et 
$U_{\LL}$ le groupe des $\lbrace \LL\rbrace$-unit\'es  de $K$.
 Posons 
 $$S=\left\lbrace  a\in U_{\LL}\ | \ 1-a\in U_{\LL}\right\rbrace.$$
 
D\'esignons   par {\bf{(FS)}} la condition suivante :

{\bf{(FS)}}  : pour tout $a\in S$, on a 
 \begin{equation}
  \label{(1.2)}
|v_{\LL}(a)|\leq 4d.
\end{equation}

Leur crit\`ere est le suivant :

\begin{theoreme} \label{T:th1}
Supposons que la condition {\bf{(FS)}} soit satisfaite par $K$. Alors, le th\'eor\`eme de Fermat asymptotique est vrai sur $K$. 
\end{theoreme}

L'ensemble $S$
est fini (\cite{Siegel}).  
Avec  les travaux de Smart, on dispose d'algorithmes permettant  d'expliciter $S$, sous r\'eserve que le groupe $U_{\LL}$ soit connu (\cite{Smart1}).
Dans ce cas, la  condition {\bf{(FS)}} est donc  en principe testable sur $K$. Cela \'etant,  la d\'etermination de $S$ n'est pas pour l'instant impl\'ement\'ee dans des logiciels de calcul et expliciter $S$ reste un  travail  g\'en\'eralement important. Par exemple, pour le sous-corps totalement r\'eel maximal $\Q\left(\mu_{16}\right)^+$ du corps cyclotomique des racines $16$-i\`emes de l'unit\'e,   $S$   est de cardinal $585$  (\cite{FS1}, 1.3).  On peut  v\'erifier  que la condition 
 {\bf{(FS)}} est satisfaite,  en particulier le th\'eor\`eme de Fermat asymptotique est vrai sur le corps $\Q\left(\mu_{16}\right)^+$ ({\it{loc. cit.}}).

Dans l'objectif de d\'emontrer le th\'eor\`eme de Fermat asymptotique sur certains corps de nombres, dans lesquels $2$ est totalement ramifi\'e,  on va introduire ci-dessous une nouvelle condition, qui d'un point de vue num\'erique a l'avantage, \`a ce jour,  de pouvoir se  tester souvent  simplement sur machine.
On \'etablit dans le th\'eor\`eme~\ref{T:th2}
qu'elle  est  \'equivalente \`a {\bf{(FS)}},  moyennant une hypoth\`ese de modularit\'e pour certaines courbes elliptiques sur $K$.

\subsection{La condition {\bf{(C)}}} Le nombre premier $2$ \'etant suppos\'e totalement ramifi\'e dans $K$, d\'esignons par  $\calH$ l'ensemble des newforms modulaires paraboliques de Hilbert  sur $K$,  de poids parall\`ele $2$ et de niveau $\LL$.  C'est un syst\`eme  fini libre sur $\C$. Pour tout $\ff\in \calH$ et tout id\'eal premier non nul $\fq$ de $O_K$,  notons   $a_{\fq}(\ff)$ le coefficient de Fourier de $\ff$ en $\fq$. C'est un entier alg\'ebrique.   Le sous-corps $\Q_{\ff}$ de $\C$ engendr\'e par les coefficients $a_{\fq}(\ff)$ est une extension finie de $\Q$.  C'est un corps totalement r\'eel ou un corps CM.  (Voir  par exemple \cite {Dembelecremona} et \cite {Dembelevoight}.)

Pour tout id\'eal premier $\fq$ de $O_K$, notons    $\Norm (\fq)$ sa norme sur $\Q$. 

La  condition    est la suivante :

 {\bf{(C)}}  : pour tout $\ff\in \calH$ tel que $\Q_{\ff}=\Q$, il existe un id\'eal premier $\fq$ de $O_K$, distinct de $\LL$,  tel que l'on ait
 \begin{equation}
 \label{(1.3)}
a_{\fq}(\ff)\not\equiv  \Norm (\fq)+1 \pmod 4.
\end{equation}

 Elle est en apparence moins simple que la condition {\bf{(FS)}}, mais 
on peut  g\'en\'eralement   la tester   en utilisant   le logiciel de calcul {\tt Magma} (\cite{MAGMA}), disons si $d\leq 8$ et si le discriminant de $K$ n'est pas trop grand. Dans le cas o\`u $d\leq 6$, 
on  dispose \'egalement de tables de  newforms d\'ecrivant $\calH$,  qui sont directement impl\'ement\'ees dans LMFDB \cite{lmfdb}.

Signalons par ailleurs que 
pour \'etablir
le th\'eor\`eme de Fermat asymptotique sur $K$ de fa\c con effective, la d\'etermination de $\calH$ est, comme on le verra, a priori indispensable dans la mise en \oe uvre de la m\'ethode modulaire.

\`A titre indicatif, le corps quadratique r\'eel de plus petit discriminant pour lequel  la condition {\bf{(C)}}  n'est pas  satisfaite est $K=\Q\left(\sqrt{114}\right)$.
Il existe une courbe elliptique sur $K$, de conducteur $\LL$, ayant tous ses points d'ordre $2$ rationnels sur $K$ (\cite{lmfdb}), ce qui explique pourquoi {\bf{(C)}}  n'est pas r\'ealis\'ee sur ce corps (th. \ref{T:th2}, lemme \ref{L:lemme1} et  \cite{FLS}) ; il en est ainsi pour une infinit\'e de corps quadratiques. Par exemple, en utilisant la table 1 de  \cite{FS1}, on peut d\'emontrer qu'il existe  une infinit\'e de corps quadratiques r\'eels $\Q\left(\sqrt{m}\right)$,  avec  $m$ sans facteurs carr\'es congru \`a $7$ modulo $8$,  pour lesquels  la condition  {\bf{(C)}}  n'est  pas    satisfaite.

Voyons un autre exemple illustrant    la condition  {\bf{(C)}}.  Soit $K$ le sous-corps totalement r\'eel maximal du corps cyclotomique $\Q(\mu_{48})$. C'est le corps  totalement r\'eel de degr\'e 8 sur $\Q$ de plus petit discriminant,   dans  lequel $2$ soit   totalement ramifi\'e (on peut le v\'erifier avec  les tables de Voight \cite{Voight}). On constate avec 
 {\tt Magma} que l'on a $|\calH|=16$ et que la condition {\bf{(C)}} est satisfaite  par $K$,  car il n'existe pas de newforms $\ff\in \calH$ telles que $\Q_{\ff}=\Q$ ; pour tout $\ff\in \calH$, on a $[\Q_{\ff}:\Q]=4$.

\subsection{Hypoth\`ese  sur le nombre de classes restreint de $K$} Notons $h_K^+$ le nombre de classes restreint de $K$. 
Rappelons que $h_K^+$ est le degr\'e sur $K$ de l'extension ab\'elienne de $K$   non ramifi\'ee   aux places finies  maximale. 

Malgr\'e de nombreux essais exp\'erimentaux, je ne suis pas parvenu \`a trouver un  exemple de corps de nombres  totalement r\'eel $K$ tel que  $2$ soit totalement ramifi\'e dans $K$ et que   $h_K^+=1$,   sans  que  la condition {\bf{(C)}}   soit   satisfaite. En particulier,   a-t-on toujours  l'implication 
$$2O_K=\LL^d \quad \hbox{et}\quad h_K^+=1 \ \Rightarrow \ {\bf{(C)}}\quad  ?$$
En fait,  la r\'eponse est positive si on a $d\in \lbrace 1,2,4,8\rbrace$. La raison \'etant que pour tout entier $n$, il existe au plus un corps totalement r\'eel $K$ pour lequel on a $d=2^n$, $2O_K=\LL^d$ et $h_K^+=1$, \`a  savoir le sous-corps totalement r\'eel maximal du corps 
 cyclotomique des racines $2^{n+2}$-i\`emes de l'unit\' e (th.~\ref{T:th5}).  
 On indiquera  par ailleurs quelques constatations  num\'eriques en faveur de cette implication.

Signalons que, comme cons\'equence du th\'eor\`eme~\ref{T:th4},  si cette implication \'etait toujours vraie, cela impliquerait le th\'eor\`eme de Fermat asymptotique sur tout corps de nombres totalement r\'eels  pour lesquels  $2O_K=\LL^d$ et $h_K^+=1$.

Les hypoth\`eses selon lesquelles  $2O_K=\LL^d$  et $h_K^+=1$ 
sont  tr\`es favorables dans l'application de la m\'ethode modulaire pour obtenir des versions effectives du th\'eor\`eme de Fermat sur $K$. Cela est notamment d\^u au fait qu'avec ces hypoth\`eses, on peut normaliser toute solution de l'\'equation de Fermat \eqref{(1.1)} de fa\c con simple (prop.~\ref{P:prop3}). On illustrera cette m\'ethode pour certains  corps de nombres de degr\'e $d\in \lbrace 3,4,5,6,8\rbrace$. On \'etablira par exemple  
que pour le corps $K=\Q\left(\mu_{16}\right)^+$,  l'ensemble $F_p(K)$ est trivial pour tout $p\geq 5$.

Tous les calculs num\'eriques que cet article a n\'ecessit\'es ont \'et\'e effectu\'es avec les logiciels de calcul {\tt Pari} (\cite{Pari})  et  {\tt Magma}.
\smallskip

{\bf{Remerciements.}} J'ai b\'en\'efici\'e de conversations avec J. Oesterl\'e concernant   l'Appendice de cet article et son application au th\'eor\`eme \ref{T:th5}. 
Je l'en remercie vivement. Je remercie \'egalement  R. Barbaud  pour l'aide qu'il m'a apport\'ee dans la r\'ealisation  de certains programmes intervenant dans le logiciel   {\tt Magma}, ainsi que D. Bernardi pour les commentaires dont il m'a fait part. 
\bigskip

{\large{{\bf{Partie 1. \'Enonc\'e des r\'esultats}}}}

Soit $K$ un corps de nombres totalement r\'eel. 

 Rappelons qu'une courbe elliptique  $E/K$ est dite modulaire s'il existe une newform modulaire parabolique de Hilbert  sur $K$, de poids parall\`ele $2$ et de niveau le conducteur de $E$,  ayant la m\^eme  fonction  L que celle de $E$.
 
 Conjecturalement,  toute courbe elliptique d\'efinie sur $K$ est modulaire.  Cela est d\'emontr\'e si  $K=\Q$ (\cite{Wiles}, \cite{Taylorwiles}, \cite{BCDT}) et  si $K$ est un corps quadratique   (\cite{FLS}). Par ailleurs, \`a $\overline{K}$-isomorphisme pr\`es, l'ensemble des  courbes elliptiques sur $K$ qui ne sont  pas modulaires est fini
 ({\it{loc. cit.}}).

\section{Les conditions  {\bf{(C)}}  et  {\bf{(FS)}} }

\begin{theoreme}  \label{T:th2} Supposons que les deux conditions  suivantes soient satisfaites :

\begin{itemize}

 \item[1)] $2$ est totalement ramifi\'e dans $K$. 

 \item[2)] Toute courbe elliptique d\'efinie sur $K$, de conducteur $\LL$, ayant tous ses points d'ordre $2$ rationnels sur $K$, est  modulaire. 
\end{itemize}

Alors, les conditions  {\bf{(C)}}  et  {\bf{(FS)}}  sont \'equivalentes.
\end{theoreme}

Si $K$ est un corps quadratique dans lequel $2$ est totalement ramifi\'e,  les conditions  {\bf{(C)}}  et  {\bf{(FS)}}  sont   donc \'equivalentes.

\section{Th\'eor\`eme de Fermat asymptotique}

Comme cons\'equence directe des   th\' eor\`emes~\ref{T:th1} et~\ref{T:th2}, on obtient l'\'enonc\'e suivant :

\begin{theoreme} \label{T:th3} Supposons que les  trois conditions  suivantes soient satisfaites :

\begin{itemize}

 \item[1)] $2$ est totalement ramifi\'e dans $K$. 

 \item[2)] Toute courbe elliptique d\'efinie sur $K$, de conducteur $\LL$, ayant tous ses points d'ordre $2$ rationnels sur $K$, est  modulaire. 
 
  \item[3)] La condition  {\bf{(C)}} est satisfaite.
  
\end{itemize}

Alors, le th\'eor\`eme de Fermat asymptotique est vrai sur $K$.
\end {theoreme}

\smallskip

\begin{remark} Dans les limites des tables  de newforms modulaires de Hilbert  sur un corps totalement r\'eel $K$ figurant dans  \cite{lmfdb},  pour lequelles   $2$ est  totalement ramifi\'e dans $K$,   on constate les donn\'ees num\'eriques suivantes. Les corps intervenant dans ces tables sont de degr\'e $d\leq 6$. 
\smallskip

1)  Pour $d=3$, il y a treize  tels corps de nombres (\`a isomorphisme pr\`es). Leurs discriminants sont 
$$148, 404, 564,756, 788, 1076,1300,1396,1492,1524,1556,1620,1940.$$
Pour chacun d'eux, la condition {\bf{(C)}} est  satisfaite. 

Cela \'etant, la condition {\bf{(C)}} n'est  pas toujours satisfaite si $d=3$. Par exemple, elle ne l'est pas  pour le corps
$K=\Q(\alpha)$ o\`u
$\alpha^3-32\alpha+2=0.$ En effet, posons $a=16\alpha$ ; 
cet entier est dans $S$ et on a $v_{\LL}(a)=13$.  La courbe elliptique $E/K$ d'\'equation
$$y^2=x(x-a)(x+1-a)$$
est de conducteur $\LL$ et a tous ses points d'ordre $2$ rationnels sur $K$ (cf. d\'em. du lemme \ref{L:lemme1}, \'equation  \eqref{(6.1)}). En utilisant le th\'eor\`eme \ref{T:th11}, on peut d\'emontrer qu'elle est modulaire. Il existe donc $\ff\in \calH$ ayant la m\^eme fonction L que celle de $E$. On a $\Q_{\ff}=\Q$ et  $\ff$ ne v\'erifie pas la condition \eqref{(1.3)},  d'o\`u notre assertion.

\smallskip

2) Pour $d=4$, il y a quinze  corps de nombres, dont les discriminants sont 
$$2048, 2304, 4352, 6224, 7168, 7488, 11344,12544,13824,14336,14656,$$
$$15952,16448,18432,18688.$$
Pour chacun de ces corps, la condition 
{\bf{(C)}} est satisfaite. Except\'e pour le corps  de discriminant $16448$,  on a   $|\calH|=0$.
 
 Signalons  \`a titre indicatif que     pour le corps $K=\Q(\alpha)$ o\`u $\alpha^4- 12\alpha^2 - 18\alpha - 5=0$, la condition {\bf{(C)}} n'est  pas  satisfaite. On peut le v\'erifier comme dans l'alin\' ea pr\'ec\'edent, en  consid\'erant l'entier  
 $a=16(1+\alpha)^2(4+4\alpha+\alpha^2)$. Il appartient \`a $S$ et  on a $v_{\LL}(a)=18$.  
On constate alors que la courbe elliptique sur $K$, d'\'equation $y^2=x(x-a)(x+1-a)$, est modulaire et que son  conducteur est $\LL$.
 
 \smallskip
 
 3) Pour $d=5$, il y a trois  corps de nombres, dont les discriminants sont 
 $$126032, 153424, 179024,$$ et ils v\'erifient la condition {\bf{(C)}}.
Pour $d=6$, il n'y a pas de tels corps   dans les tables de \cite{lmfdb}.
\end{remark}

Dans le cas o\`u  l'on a $h_K^+=1$,  on peut s'affranchir  de l'hypoth\`ese de modularit\'e :

\smallskip

\begin{theoreme} \label{T:th4}Supposons que les  trois conditions  suivantes soient satisfaites :

\begin{itemize}

 \item[1)] $2$ est totalement ramifi\'e dans $K$. 

 \item[2)]  On a $h_K^+=1$.
 
  \item[3)] La condition  {\bf{(C)}} est satisfaite.
  
\end{itemize}

Alors, le th\'eor\`eme de Fermat asymptotique est vrai sur $K$.
\end {theoreme}

\section{Question}
  
Certaines  constatations num\'eriques concernant les hypoth\`eses faites dans le th\'eor\`eme~\ref{T:th4} sugg\`erent la question suivante :

\begin{question}    Supposons 
$2$  totalement ramifi\'e dans $K$ 
et  $h_K^+=1$.
La condition  {\bf{(C)}}  est-elle toujours satisfaite ?
\end{question}

\begin{proposition} \label{P:prop1}
La r\'eponse est positive si on a  $d\in \left\lbrace 1,2,4,8\right\rbrace$.
\end{proposition}

C'est une cons\'equence du r\'esultat qui suit. Sa d\'emonstration  repose sur  la premi\`ere assertion du th\'eor\`eme \ref{T:th17} de l'Appendice  (Partie 5). 
Pour tout $n\geq 1$, soit  $\mu_{2^{n+2}}$ le groupe des racines $2^{n+2}$-i\`emes de l'unit\'e. Notons  
$\Q(\mu_{2^{n+2}})^+$
le sous-corps totalement r\'eel maximal de $\Q(\mu_{2^{n+2}})$.  
\medskip

\begin{theoreme} \label{T:th5} Soient $n$ un entier et $K$ un corps de nombres totalement r\'eel,  de degr\' e $2^n$ sur $\Q$,  satisfaisant   les  conditions suivantes :

\begin{itemize}

 \item[1)] $2$ est totalement ramifi\'e dans $K$. 

 \item[2)]  On a $h_K^+=1$.

\end{itemize}

 Alors, on a $K=\Q(\mu_{2^{n+2}})^+$.
\end {theoreme} 
 
Le corps  $\Q(\mu_{2^{n+2}})^+$ satisfait la premi\`ere condition. Pour $n\leq 5$,  son nombre de classes restreint vaut $1$. On conjecture   que pour tout $n$, son nombre de classes   vaut $1$. Certains r\'esultats r\'ecents ont \'et\'e d\'emontr\'es dans cette direction (\cite{FK}).

On peut  v\'erifier directement avec  {\tt Magma} que pour $n\leq 3$,  la condition {\bf{(C)}} est satisfaite  pour   $\Q(\mu_{2^{n+2}})^+$, ce qui implique la proposition~\ref{P:prop1}. 
   
\medskip
  
{\bf{Faits exp\'erimentaux.}} Indiquons quelques constatations num\'eriques en faveur d'une r\'eponse positive \`a la question 4.1. En utilisant les tables de Voight  (\cite{Voight}), j'ai dress\'e une liste de corps totalement r\'eels  pour   lesquels :

\begin{itemize}

 \item[1)] $d\in \left\lbrace 3,5,6,7\right\rbrace$,
 
 \item[2)] $2$ est totalement ramifi\'e dans $K$,

 \item[3)]    $h_K^+=1$, 
  
 \item[4)] le discriminant $D_K$ de $K$ est pair plus petit qu'une borne   fix\'ee.

\end{itemize}

Dans la table 1 ci-dessous, l'entier    $N$ est le nombre de  corps totalement r\'eels de degr\'e $d$ et de discriminant $D_K$ pair  plus petit que la  borne que l'on s'est fix\'ee (\`a isomorphisme pr\`es). Dans la derni\`ere colonne se trouve le nombre de corps pour lesquels $2$ est totalement ramifi\'e et $h_K^+=1$.

\begin{table}[htb]
\centerline{\vbox{\offinterlineskip
\halign{\vrule height12pt\ \hfil#\hfil\ \vrule&&\ \hfil#\hfil\ \vrule\cr
\noalign{\hrule}
\omit\vrule height2pt\hfil\vrule&&&\cr
$d$&Borne sur $D_K$ &$N$&  $2O_K=\LL^d$ \ et\ $h_K^+=1$\cr
\omit\vrule height2pt\hfil\vrule&&&\cr
\noalign{\hrule}
3&$21.10^3$& 378& 80  \cr
\omit\vrule height2pt\hfil\vrule&&& \cr
\noalign{\hrule}
5&$17.10^5$& 315&23 \cr
\omit\vrule height2pt\hfil\vrule&&&\cr
\noalign{\hrule}
6&$21.10^6$& 361&7 \cr
\omit\vrule height2pt\hfil\vrule&&&\cr
\noalign{\hrule}
7&$207.10^6$& 32&2 \cr
\omit\vrule height2pt\hfil\vrule&&&\cr
\noalign{\hrule}
}
}}
\caption{ }
\label{Table}
\end{table}

Il y a  cent-douze corps $K$ intervenant dans ce tableau pour lesquels $2O_K=\LL^d$ et $h_K^+=1$. 
Pour chacun d'eux,  on constate avec  {\tt Magma} que la condition {\bf{(C)}} est satisfaite.

Les discriminants    de ces cent-douze corps  sont explicit\'es   ci-dessous. Des \'el\'ements primitifs de chacun de ces  corps  sont d\'etermin\'es   dans les tables de Voight.

\bigskip

Pour $d=3$ :

\begin{table}[htb]
\centerline{\vbox{\offinterlineskip
\halign{\vrule height12pt\ \hfil#\hfil\ \vrule&&\ \hfil#\hfil\ \vrule\cr
\noalign{\hrule}
\omit\vrule height2pt\hfil\vrule&&&&&&&&&&&\cr
148  &404 & 564 & 756& 1300&1524&1620&2228&2708&2804&3124&3252 \cr
\omit\vrule height2pt\hfil\vrule&&&&&&&&&&& \cr
\noalign{\hrule}
3508 &3540  & 3604&3892&4628&4692& 4852&5172&5204&5940&6420&7028\cr
\omit\vrule height2pt\hfil\vrule&&&&&&&&&&&\cr
\noalign{\hrule}
7668  &7700  & 7796& 8308& 8372&8628&8692&9044&9076&9204&9300&9460\cr
 \omit\vrule height2pt\hfil\vrule&&&&&&&&&&&\cr
\noalign{\hrule}
9812 & 10164 & 10260 &10292& 10324& 10580& 10868&11060&11092&11476&12660&12788\cr
  \omit\vrule height2pt\hfil\vrule&&&&&&&&&&&\cr
\noalign{\hrule}
12852 & 13172 &13684&13748& 13972&14420&14516&14964&15252&15284&15380&15444\cr
\omit\vrule height2pt\hfil\vrule&&&&&&&&&&&\cr
\noalign{\hrule}
15700 & 16084&16116&16180& 16532&17556&17684&17716&17780&18292&18644&18740\cr
\omit\vrule height2pt\hfil\vrule&&&&&&&&&&&\cr
\noalign{\hrule}
19252 & 19348&19572&20276& 20436&20724&20788&20948&&&&\cr
\omit\vrule height2pt\hfil\vrule&&&&&&&&&&&\cr
\noalign{\hrule}
 }
}}
\caption{ }
\label{Table}
\end{table}

Pour $d=5$ :

\begin{table}[htb]
\centerline{\vbox{\offinterlineskip
\halign{\vrule height12pt\ \hfil#\hfil\ \vrule&&\ \hfil#\hfil\ \vrule\cr
\noalign{\hrule}
\omit\vrule height2pt\hfil\vrule&&&&&&&&\cr
126032 & 153424 & 179024 & 207184 & 223824 & 394064 & 453712 & 535120 & 629584 \cr
\omit\vrule height2pt\hfil\vrule&&&&&&&& \cr
\noalign{\hrule}
708944 & 747344& 970448 & 981328 & 1034192 & 1104464 & 1172304& 1197392& 1280592 \cr
\omit\vrule height2pt\hfil\vrule&&&&&&&&\cr
\noalign{\hrule}
 1284944& 1395536 & 1550288 & 1664592 & 1665360 &   &  &   &   \cr
 \omit\vrule height2pt\hfil\vrule&&&&&&&&\cr
\noalign{\hrule}
 }
}}
\caption{ }
\label{Table}
\end{table}

 Pour $d=6$ :  
 
 \begin{table}[htb]
\centerline{\vbox{\offinterlineskip
\halign{\vrule height12pt\ \hfil#\hfil\ \vrule&&\ \hfil#\hfil\ \vrule\cr
\noalign{\hrule}
\omit\vrule height2pt\hfil\vrule&&&&&&\cr
2803712& 4507648 & 5163008 & 6637568 & 7718912 & 10766336 & 20891648  \cr
 \omit\vrule height2pt\hfil\vrule&&&&&&\cr
\noalign{\hrule}
 }
}}
\caption{ }
\label{Table}
\end{table}
Pour $d=7$ : on a  $D_K\in \left\lbrace 46643776, 196058176\right\rbrace$.

\medskip

On en d\'eduit avec le  th\'eor\'eme~\ref{T:th4}  l'\'enonc\'e  suivant  :

\begin{proposition} \label{P:prop2} Pour chacun des   corps de nombres $K$   indiqu\'es    ci-dessus, 
le th\'eor\`eme de Fermat asymptotique est vrai sur  $K$.
 \end{proposition}

Signalons que, pour $d\neq 6$, le groupe de Galois sur $\Q$ de la cl\^oture 
galoisienne de $K$ est isomorphe \`a $\sym_d$. Pour $d=6$, 
il est non ab\'elien d'ordre $12$ ou $72$.

\section{Th\'eor\`eme de Fermat effectif - Exemples} 
\label{S:5}
On \'etablit   le th\'eor\`eme de Fermat asymptotique de fa\c con effective  pour quelques  corps de nombres  figurant dans ces tables, ainsi que pour les corps  
$\Q\left(\mu_{16}\right)^+$ et $\Q\left(\mu_{32}\right)^+.$

\begin{theoreme} \label{T:th6} Soit $K$ un  corps cubique   r\'eel de discriminant $D_K\in \big\lbrace 148, 404, 564\big\rbrace$. 
Pour tout  $p\geq 5$, l'ensemble $F_p(K)$ est trivial.
\end {theoreme}

\medskip

\begin{theoreme} \label{T:th7}  Posons  $K=\Q(\alpha)$ avec 
\begin{equation} 
\label{(5.1)}
\alpha^5-6\alpha^3+6\alpha-2=0.
\end{equation}
C'est le  corps   de plus petit discriminant  intervenant dans la table 3.
Pour tout  $p$ distinct de $5,13,17,19$,  l'ensemble   $F_p(K)$  est trivial.
\end {theoreme}

\medskip

\begin{theoreme} \label{T:th8} Posons  $K=\Q(\alpha)$ avec 
\begin{equation} 
\label{(5.2)}
\alpha^6+2\alpha^5-11\alpha^4-16\alpha^3+15\alpha^2+14\alpha-1=0.
\end{equation}
C'est le  corps   de plus petit discriminant  intervenant dans la table 4.
Pour tout  $p\geq 29$, distinct de $37$,   l'ensemble    $F_p(K)$  est trivial.
\end {theoreme}

\medskip

\begin{theoreme} \label{T:th9} 

1) Pour tout $p\geq 5$, l'ensemble  $F_p\left(\Q\left(\mu_{16}\right)^+\right)$  est trivial.

2) Pour tout $p>6724$,  l'ensemble  $F_p\left(\Q\left(\mu_{32}\right)^+\right)$  est trivial.
\end {theoreme}

On a $6724=\left(1+3^4\right)^2$, qui  est, pour $d=8$,  la borne obtenue dans \cite{Oesterle} concernant  les points de $p$-torsion des courbes elliptiques sur   les corps de nombres de degr\'e $d$.
 
 \bigskip
 
{\large{{\bf{Partie 2. Les  th\'eor\`emes 2  et 5}}}}

Pour tout id\'eal premier $\fq$ de $O_K$, notons $v_{\fq}$ la valuation sur $K$ qui lui est associ\' ee, et pour toute courbe elliptique $E/K$,  notons $j_E$ son invariant modulaire.

\section{D\'emonstration du th\'eor\`eme \ref{T:th2}}

\begin{lemme} \label{L:lemme1} Les trois assertions suivantes sont \'equivalentes :

1)  La condition {\bf{(FS)}} est satisfaite.

2) Il n'existe pas de courbe elliptique $E/K$ telle que l'on ait :

\begin{itemize}
   \item[(1)]    $v_{\LL}(j_E)<0$, 

   \item[(2)]    pour tout id\'eal premier $\fq$ de $O_K$, distinct de $\LL$, on a $v_{\fq}(j_E)\geq 0$,
  
   \item[(3)]  $E$ a tous ses points d'ordre $2$ rationnels sur $K$. 
\end {itemize}

3)   Il n'existe pas de courbe  elliptique sur $K$, de conducteur $\LL$, ayant tous ses points d'ordre $2$ rationnels sur $K$.

\end{lemme}

\begin{proof} V\'erifions  l'implication $1)\Rightarrow 2) $. Supposons pour cela qu'il existe une courbe elliptique $E/K$ satisfaisant  les trois conditions  de la deuxi\`eme assertion. D'apr\`es la condition (3), \`a torsion quadratique pr\`es, il existe $\lambda\in K$ tel que $E/K$ poss\`ede une \'equation de la forme de Legendre (\cite{Silverman}, p. 49, prop. 1.7, assertion (a) et sa d\'emonstration)
$$y^2=x(x-1)(x-\lambda).$$
Posons 
$$\mu=1-\lambda.$$
On a les \'egalit\'es
$$j_E=2^8\frac{(\lambda^2-\lambda+1)^3}{\lambda^2(1-\lambda)^2}=2^8 \frac{(1-\lambda \mu)^3}{(\lambda \mu)^2}.$$
Soit $O_{\LL}$ l'anneau des $\lbrace \LL\rbrace$-entiers de $K$. D'apr\`es la condition (2),  on a 
$$j_E\in O_{\LL}.$$
De plus,  $\lambda$, $\frac{1}{\lambda}$, $\mu$ et $\frac{1}{\mu}$
sont  racines d'un polyn\^ome unitaire  (de degr\'e $6$)  \`a coefficients dans $O_{\LL}$. Par suite, 
$\lambda$ et $\mu$   appartiennent \`a $S$.
 Posons
$$t=\Max\left( |v_{\LL}(\lambda)|,|v_{\LL}(\mu)|\right).$$
Il r\'esulte de la condition (1) que l'on a  $t>0$. L'\'egalit\'e $\lambda+\mu=1$ implique alors que l'on a
$$v_{\LL}(\lambda)=v_{\LL}(\mu)=-t\quad \hbox{ou}\quad v_{\LL}(\lambda)=0, \ v_{\LL}(\mu)=t \quad \hbox{ou}\quad v_{\LL}(\lambda)=t, \ v_{\LL}(\mu)=0.$$
On obtient dans tous les cas
$$v_{\LL}(j_E)=8v_{\LL} (2)-2t=8d-2t.$$
La condition (1) implique   $t>4d$ et donc la condition {\bf{(FS)}} n'est pas satisfaite (in\'egalit\'e \eqref{(1.2)}). Cela prouve la premi\`ere implication.

L'implication $2)\Rightarrow 3)$ est imm\'ediate.

 D\'emontrons l'implication $3)\Rightarrow 1)$. Supposons la condition ${\bf{(FS)}}$  non satisfaite, autrement dit qu'il existe   $a\in S $  tel que l'on ait
$$|v_{\LL}(a)|> 4d.$$
Posons $b=1-a$.

Supposons  $v_{\LL}(a)>4d$.  V\'erifions   que la courbe elliptique $E/K$ d'\'equation
\begin{equation}
\label{(6.1)}
y^2=x(x-a)(x+b)
\end{equation}
est de conducteur $\LL$, ce qui \'etablira l'implication dans ce cas. Notons $c_4(E)$ et $\Delta(E)$ les invariants standard associ\'es \`a cette \'equation. 
On a 
$$c_4(E)=16(a^2+ab+b^2)\quad \hbox{et}\quad \Delta(E)=16(ab)^2.$$
Pour tout id\'eal premier $\fq$ de $O_K$ distinct de $\LL$, on a 
$ v_{\fq}\left(\Delta(E)\right)=0,$
donc $E/K$ a bonne r\'eduction en $\fq$. 
Posons 
$$x=4X\quad \hbox{et}\quad y=8Y+4X.$$
On obtient comme nouveau mod\`ele de $E/K$
$$(W) : Y^2+XY=X^3-{a\over 2} X^2-{ab\over 16} X.$$
 On a $v_{\LL}(a)>4d=4v_{\LL}(2)$, donc ce mod\`ele est entier. Par ailleurs, on a 
$$v_{\LL}\left(c_4(E)\right)=4d\quad \hbox{et}\quad v_{\LL}\left(\Delta(E)\right)=4d+2v_{\LL}(a)>12d,$$
$$c_4(E)=2^4c_4(W)\quad \hbox{et}\quad \Delta(E)=2^{12}\Delta(W),$$
d'o\`u
$$v_{\LL}\left(c_4(W)\right)=0\quad \hbox{et}\quad v_{\LL}\left(\Delta(W)\right)>0.$$
Ainsi $E$ a r\'eduction de type multiplicatif en $\LL$, d'o\`u notre assertion.

Supposons  $v_{\LL}(a)<-4d$. 
Posons 
$$a'=\frac{1}{a}\quad \hbox{et}\quad b'=1-a'.$$
On a $b'=-b/a$ donc  $a'$ est dans $S$ 
et on a  
$v_{\LL}(a')>4d$.
Comme ci-dessus, on v\'erifie   que la courbe elliptique d'\'equation 
$y^2=x(x-a')(x+b')$
est de conducteur $\LL$. 
Cela \'etablit l'implication.
\end{proof}

\begin{lemme} \label{L:lemme2} Supposons que toute courbe elliptique sur $K$, de conducteur $\LL$, ayant tous ses points d'ordre $2$ sur $K$, soit modulaire. 
Les deux assertions  suivantes sont \'equivalentes :

1)   Il n'existe pas de courbe  elliptique sur $K$, de conducteur $\LL$, ayant tous ses points d'ordre $2$ rationnels sur $K$.
   
2)  La condition {\bf{(C)}} est satisfaite.

\end{lemme}

\begin{proof} Pour tout id\'eal premier  $\fq\neq \LL$ de $O_K$, posons $\F_{\fq}=O_K/\fq$.

Supposons  que la premi\`ere condition soit r\'ealis\'ee. V\'erifions que la seconde l'est aussi. 
Soit $\ff$ une newform  de $\calH$ telle que $\Q_{\ff}=\Q$. Proc\'edons par l'absurde en supposant que pour tout id\'eal premier  $\fq\neq \LL$ de $O_K$ la condition \eqref{(1.3)} ne soit pas satisfaite.
Parce que le niveau de $\ff$ est $\LL$, il existe une courbe elliptique $E/K$, de conducteur $\LL$,  ayant la m\^eme fonction L que celle de $\ff$ (\cite{FS1}, th. 8).
Pour tout id\'eal premier  $\fq\neq \LL$ de $O_K$,   l'ordre de $E(\F_{\fq})$ est  donc multiple de $4$. 
Il en r\'esulte que 
$E/K$ est  li\'ee par une isog\'enie de degr\'e $\leq 2$ \`a une courbe elliptique $F/K$ 
ayant tous ses points d'ordre 2 sur $K$ (\cite{Sengunsiksek}, lemme 7.5). Le conducteur de $F/K$ est $\LL$, d'o\`u  une contradiction.  (On n'a pas utilis\' e  ici l'hypoth\`ese de modularit\'e.)

Inversement,   supposons  qu'il existe une  courbe elliptique $E/K$, de conducteur $\LL$,  ayant tous ses points d'ordre 2 sur $K$. Par hypoth\`ese,  $E$ \'etant modulaire, il existe une newform $\ff\in \calH$ ayant la m\^eme fonction L que celle de $E$. On a $\Q_{\ff}=\Q$.  Soit $\fq$ un id\'eal premier de $O_K$ distinct de $\LL$. La courbe elliptique $E$ a bonne r\'eduction en  $\fq$ et l'application $E(K)[2]\to E(\F_{\fq})$ est   injective  (cf. \cite{Silverman},  p. 192, prop. 3.1). Par suite, $4$ divise l'ordre de $E(\F_{\fq})$, autrement dit, on a   $ a_{\fq}(\ff) \equiv \Norm(\fq)+1\pmod 4$. Ainsi, la condition ${\bf{(C)}}$ n'est pas satisfaite, d'o\`u le r\'esultat.
\end{proof}

Le th\'eor\`eme~\ref{T:th2} est  une cons\'equence directe des lemmes~\ref{L:lemme1} et~\ref{L:lemme2}.

\begin{remark} L'hypoth\`ese de modularit\'e est intervenue dans la d\'emonstration pour \'etablir   l'implication 
${\bf{(C)}}\Rightarrow {\bf{(FS)}} $.
\end{remark}

\section{D\'emonstration du  th\'eor\`eme \ref{T:th5}}

On d\'emontre par r\'ecurrence que pour tout $r\geq 1$, tel que $r\leq n+2$, on a l'inclusion
\begin{equation}
\label{(7.1)}
\Q\left(\mu_{2^r}\right)^+\subseteq K.
\end{equation}
Cela \'etablira le r\'esultat,  car   $\Q\left(\mu_{2^{n+2}}\right)^+$ et $K$ sont de m\^eme degr\'e $2^n$ sur $\Q$. 
L'inclusion \eqref{(7.1)}  est vraie si $r=1$ et $r=2$. Soit $r$ un entier tel que $2\leq r<n+2$ et que \eqref{(7.1)} soit vraie. Il s'agit de v\'erifier que 
$\Q\left(\mu_{2^{r+1}}\right)^+$ est contenu dans $K$. Posons 
$$L=K\Q\left(\mu_{2^{r+1}}\right)^+.$$
L'extension $L/K$ est non ramifi\'ee en dehors des id\'eaux premiers de $O_K$ au-dessus de $2$, y compris aux places \`a l'infini. Par suite, son conducteur est une puissance de $\LL$. Plus pr\'ecis\'ement :

\begin{lemme} \label{L:lemme3} Le conducteur de $L/K$ divise $4O_K$.
\end{lemme}
\begin{proof}  Soit $\zeta$ une racine primitive $2^{r+1}$-i\`eme de l'unit\'e.  
Il  existe $x\in O_K$ tel que $x$ appartienne \`a $\LL$ et pas \`a $\LL^2$. On a $r-1\leq n$. Posons 
$$a=x^{2^{n-(r-1)}}\quad \hbox{et}\quad u=\left(\frac{\zeta+\zeta^{-1}}{a}\right)^2.$$
On a   
$(\zeta+\zeta^{-1})ˆ^2=\zeta^2+\zeta^{-2}+2.$
Parce que $\zeta^2$ est une racine primitive $2^r$-i\`eme de l'unit\'e,  on d\'eduit de   \eqref{(7.1)} que $u$ appartient \`a $K$. 
Par  ailleurs, on a 
$$[\Q\left(\mu_{2^{r+1}}\right)^+:\Q\left(\mu_{2^{r}}\right)^+]=2 \quad \hbox{et}\quad  \Q\left(\mu_{2^{r+1}}\right)^+=\Q\left(\zeta+\zeta^{-1}\right).$$
Il en r\'esulte que l'on a 
$[L:K]\leq 2,$
puis l'\'egalit\'e 
$$L=K\left(\sqrt{u}\right).$$
On a 
\begin{equation}
\label{(7.2)}
v_{\LL}(u)=0.
\end{equation}
En effet,   $v$ \'etant la valuation de $\overline{\Q_2}$ normalis\' ee par $v(2)=1$, on a 
$$v(x)=\frac{1}{2^n},\quad v(a)=\frac{1}{2^{r-1}}\quad \hbox{et}\quad v\left(\zeta+\zeta^{-1}\right)=\frac{1}{2^{r-1}},$$
d'o\`u \eqref{(7.2)}. Ainsi,  le discriminant de $L/K$ divise $4O_K$. On obtient le r\'esultat car  le conducteur et le discriminant de $L/K$   sont \'egaux (\cite{Cassels}, p. 160, si $[L:K]=2$).
\end{proof}

 Soit $K^{4O_K}$  le corps de classes de rayon modulo $4O_K$ sur $K$. D'apr\`es le lemme pr\'ec\'edent,  $L$ est contenu dans $K^{4O_K}$.  Par hypoth\`ese,  $2$ est totalement ramifi\'e dans $K$ et 
 on a $h_K^+=1$.
 D'apr\`es l'assertion 1 du  th\'eor\`eme \ref{T:th17} de l'Appendice, on a donc $K^{4O_K}=K$. On obtient $L=K$, ce qui montre que $\Q\left(\mu_{2^{r+1}}\right)^+$ est contenu dans $K$, d'o\`u le   th\'eor\`eme.

\bigskip
\medskip

{\large{{\bf{Partie 3. La m\'ethode modulaire}}}}

Les d\'emonstrations  du th\'eor\`eme~\ref{T:th4} et des r\'esultats annonc\'es dans le paragraphe \ref{S:5}, reposent sur  la m\'ethode modulaire, analogue \`a celle 
utilis\'ee par Wiles pour \'etablir le th\'eor\`eme de Fermat sur $\Q$.
On peut trouver dans  \cite{FS1} un expos\'e  d\'etaill\'e de cette m\'ethode.
Le principe g\'en\'eral consiste \`a proc\`eder par l'absurde en supposant qu'il existe un point non trivial dans $F_p(K)$. 
On lui associe ensuite   une courbe elliptique d\'efinie sur $K$ et en \'etudiant  le  module galoisien de ses points de $p$-torsion, on essaye  
d'obtenir une contradiction.

D\'ecrivons   la mise en \oe uvre de cette m\'ethode  dans notre situation.
Soit $K$ un corps de nombres totalement r\'eel, de degr\'e $d$ sur $\Q$, satisfaisant les deux conditions suivantes :

\begin{itemize}

 \item[1)]  $2$ est totalement ramifi\'e dans $K$. 

 \item[2)]  On a $h_K^+=1$. 

\end{itemize}

\section{La courbe elliptique $E_0/K$} 
Consid\'erons un point  $(a,b,c)\in F_p(K)$ tel que $abc\neq 0$.
On peut supposer  que l'on a 
\begin{equation} 
\label{(8.1)}
a,b,c\in O_K.
\end{equation}
On a $h_K^+=1$, en particulier  $O_K$ est principal. On supposera d\'esormais, cela n'est pas restrictif, que l'on a 
\begin{equation} 
\label{(8.2)}
aO_K+bO_K+cO_K=O_K.
\end{equation}
Soit $E_0/K$ la cubique, appel\'ee souvent courbe de Frey, d'\'equation
\begin{equation} 
\label{(8.3)}
y^2=x(x-a^p)(x+b^p).
\end{equation}
Les invariants standard  qui lui sont associ\'es  sont
\begin{equation}
\label{(8.4)}
c_4(E_0)=16\bigl(a^{2p}+(ab)^p+b^{2p}\bigr),  \quad c_6(E_0)=-32(a^p-b^p)(b^p-c^p)(c^p-a^p),   
 \end{equation}
\begin{equation}
\label{(8.5)}
\Delta(E_0)=16(abc)^{2p}.
 \end{equation}
En particulier, $E_0$ est une courbe elliptique d\'efinie sur $K$. 

\subsection{R\'eduction de  $E_0/K$}

\begin{lemme}  \label{L:lemme4} Soit $\fq$ un id\'eal premier de $O_K$ distinct de $\LL$. 

1)  L'\'equation \eqref{(8.3)} est minimale en $\fq$. 

2) Si $\fq$ ne divise pas $abc$, $E_0$ a bonne r\'eduction en $\fq$. 

3) Si $\fq$ divise  $abc$, $E_0$ a r\'eduction de type multiplicatif en $\fq$. 
\end{lemme}  

\begin{proof}  C'est une cons\'equence de la condition \eqref{(8.1)} ainsi que des formules \eqref{(8.2)}, \eqref{(8.4)} et \eqref{(8.5)}.
\end{proof}

Pour tout id\'eal premier $\fq$ de $O_K$, notons   $\Delta_{\fq}$ un discriminant local minimal de $E_0$ en $\fq$.

\begin{proposition}  \label{P:prop3} Supposons $p>4d$. Quitte \`a multiplier $(a,b,c)$ par une unit\'e convenable de $O_K$,  les deux conditions suivantes sont satisfaites :

1)   $E_0$ a r\' eduction de type multiplicatif en $\LL$. 

2)  On a 
$$v_{\LL}(\Delta_{\LL})=2pv_{\LL}(abc)-8d.$$

En particulier, avec une telle  normalisation, $E_0/K$ est semi-stable. 

\end {proposition}

\begin{proof} Le nombre premier $2$ \'etant   totalement ramifi\'e dans $K$, on a  $O_K/\LL=\F_2$. L'un des entiers $a,b,c$ est donc divisible par $\LL$. On peut  supposer que $\LL$  divise $b$ et que $\LL$ ne divise pas $ac$.  On a de plus    $h_K^+=1$, donc le corps de classes de rayon modulo $4O_K$ sur $K$ est   \'egal \`a  $K$  (th. \ref{T:th17}). Soit $U_K$ le groupe des unit\'es de $O_K$. 
Le morphisme naturel $U_K\to (O_K/4O_K)^* $ est   surjectif (lemme \ref{L:lemme16}). 
 Il existe  ainsi $\varepsilon\in U_K$ tel que l'on ait
$$\varepsilon^{-1}\equiv -a \pmod 4.$$
Posons 
$$a'=\varepsilon a, \quad b'=\varepsilon b, \quad c'=\varepsilon c.$$
Consid\'erons  alors la courbe elliptique $E'_0/K$ d'\'equation
\begin{equation} 
\label{(8.6)}
y^2=x(x-a'^p)(x+b'^p).
\end{equation}
En effectuant le changement de variables 
$$x=4X\quad \hbox{et}\quad y=8Y+4X,$$
on obtient comme nouveau mod\`ele 
$$(W) : Y^2+XY=X^3+ \left( \frac{b'^p-a'^p-1}{4}\right) X^2-\frac{(a'b')^p}{16} X.$$
 On a 
$$a'^p+1\equiv 0 \pmod 4,$$
et d'apr\`es l'hypoth\`ese faite sur $p$, 
$$v_{\LL}(b'^p)=pv_{\LL}(b')\geq p>4d=4v_{\LL}(2).$$
Par suite, $(W)$ est un  mod\`ele est entier. 
En notant  $c_4(W)$ et $\Delta(W)$ les invariants standard qui lui sont associ\'es,
on a 
$$c_4(E'_0)=2^4c_4(W)\quad \hbox{et}\quad \Delta(E'_0)=2^{12}\Delta(W).$$
D'apr\`es les formules \eqref{(8.4)} et \eqref{(8.5)}, utilis\'ees avec l'\'equation \eqref{(8.6)}, on obtient 
$$v_{\LL}(c_4(W))=0\quad \hbox{et}\quad v_{\LL}(\Delta(W))=2pv_{\LL}(a'b'c')-8d>0.$$
Ainsi, $(W)$ est un mod\`ele minimal de $E'_0/K$, qui a donc r\'eduction de type multiplicatif en $\LL$. Parce que $\varepsilon$ est une unit\'e de $O_K$,  et compte tenu du lemme~\ref{L:lemme4},  cela entra\^ine le r\'esultat.
\end{proof}

Dans le cas o\`u   $p>4d$, on supposera,  dans toute la suite,   que le triplet $(a,b,c)\in F_p(K)$ est normalis\'e  de sorte que les deux conditions de la proposition~\ref{P:prop3}  soient satisfaites.

\subsection{Modularit\'e de  $E_0/K$} 
D'apr\`es le corollaire 2.1  de  \cite{FS1} :

\begin{theoreme} \label{T:th10} La courbe elliptique $E_0/K$ est modulaire si $p$ est plus grand qu'une constante  qui ne d\'epend que de $K$. 
\end{theoreme} 

D'apr\`es   la remarque qui suit le corollaire 2.1  de  {\it{loc. cit.}},  ce r\'esultat n'est pas effectif en g\'en\'eral. 
Cependant l'\'enonc\'e suivant    permet parfois de d\'emontrer qu'une  elliptique semi-stable  d\'efinie sur $K$ est modulaire (\cite{FLS} th. 7) :

\begin{theoreme} \label{T:th11} Posons $\ell=5$ ou  $\ell=7$. Supposons qu'il existe un id\'eal premier   de $O_K$ au-dessus de $\ell$ en lequel  l'extension $K/\Q$ soit non ramifi\'ee. Soit $E/K$ une courbe elliptique semi-stable sur $K$. Si $E(\overline K)$  n'a pas de sous-groupe d'ordre $\ell$  stable par $\Gal(\overline K/K)$, alors $E/K$ est modulaire.
\end{theoreme}

\section{La repr\'esentation $\rho_{E_0,p}$}    Notons
$$\rho_{E_0,p} : \Gal(\overline{K}/K)\to \Aut(E_0[p])\simeq \GL_2(\F_p)$$
la repr\'esentation donnant l'action de $\Gal(\overline{K}/K)$ sur le groupe  des points de $p$-torsion de $E_0$.

\subsection{Le conducteur de $\rho_{E_0,p}$} Notons $N_{E_0}$ le conducteur de $E_0/K$.
Posons   
$$M_p=\prod_{\fq\mid N_{E_0} \atop p\mid v_{\fq}(\Delta_{\fq})} \fq \quad \hbox{et}\quad N_p=\frac{N_{E_0}}{M_p}.$$

\begin{lemme} \label{L:lemme5} Supposons $p>4d$. On a $N_p=\LL$.
\end{lemme}

D\'emonstration : D'apr\`es le lemme \ref{L:lemme4} et la formule \eqref{(8.5)}, pour tout id\'eal premier $\fq$ de $O_K$, distinct de $\LL$, on a 
$$v_{\fq}(\Delta_{\fq})\equiv 0 \pmod p.$$
La seconde condition   de la proposition \ref{P:prop3} entra\^ine alors le r\'esultat.

\begin{remark} La terminologie adopt\'ee dans ce paragraphe se justifie par le fait que si l'on a $p>4d$, on peut d\'emontrer que $\LL$ le conducteur de Serre de 
$\rho_{E_0,p}$ (cf. \cite{Se2} pour $K=\Q$).
\end{remark}

\subsection{Irr\'eductibilit\'e de $\rho_{E_0,p}$} Le corps $K$ \'etant totalement r\'eel, il ne contient pas le corps de classes de Hilbert d'un corps quadratique imaginaire. D'apr\`es  la proposition de l'Appendice B de \cite{Kraus2007}, on a ainsi l'\'enonc\'e suivant :

\begin{theoreme} \label{T:th12} La repr\'esentation  $\rho_{E_0,p}$ est irr\'eductible si  $p$ est plus grand qu'une constante que ne d\'epend que de $K$. 
\end{theoreme} 

En ce qui concerne l'effectivit\'e de cet \'enonc\'e, 
consid\'erons plus g\'en\'eralement dans la suite de ce paragraphe une courbe elliptique $E/K$ semi-stable. Notons $\rho_{E,p}$ la repr\'esentation  donnant l'action  $\Gal(\overline{K}/K)$ 
sur son groupe   des points de $p$-torsion. 
Rappelons un crit\`ere permettant souvent d'\'etablir de mani\`ere effective que  $\rho_{E,p}$ est irr\'eductible (cf. \cite{Kraus2007}).

Soit $p_0$ le plus grand nombre premier 
pour lequel il existe une courbe elliptique  d\'efinie sur $K$ ayant un point d'ordre $p_0$ rationnel sur $K$. Il est born\'e par une fonction de $d$ (\cite{Merel96}) ;  plus pr\'ecis\'ement, on a (\cite{Oesterle})
$$p_0\leq \left(1+3^{{d\over 2}}\right)^2.$$
Notons  $[{d\over 2}]$  la partie enti\`ere de $\frac{d}{2}$. Soit $U_K^+$ le groupe des unit\'es totalement positives de $O_K$. 
Pour tout $u\in U_K^+$  et tout entier $n$ tel que $1\leq n\leq [{d\over 2}]$,  on 
d\'efinit  le polyn\^ome
$H_n^{(u)}\in \Z[X]$  comme suit. Soient $H$  le polyn\^ome minimal de $u$ sur $\Q$ et $t$ son  degr\'e.  On pose
\begin{equation}
\label{(9.2)}
H_1^{(u)}=H\quad \hbox{et}\quad G=X^tH\left({Y\over X}\right)\in \Z[Y][X].
\end{equation}
Pour tout $n\geq 2$,    $H_n^{(u)}$ est 
le polyn\^ome de $\Z[X]$ obtenu en substituant $Y$ par $X$ dans
\begin{equation}
\label{(9.3)}
 \Res_X\left(H_{n-1}^{(u)},G\right)\in \Z[Y],
 \end{equation}
 le r\'esultant par rapport \`a $X$ de $H_{n-1}^{(u)}$ et $G$.
Il est unitaire de degr\'e $t^n$ et ses racines sont les produits de $n$ racines de $H$ compt\'ees avec multiplicit\'es. Posons 
\begin{equation}
\label{(9.4)}
A_n=\pgcd_{u\in U_K^+} H_n^{(u)}(1) \quad \hbox{et}\quad  R_K=\prod_{n=1}^{[{d\over 2}]} A_n.
 \end{equation}

L'\'enonc\'e qui suit est une reformulation du th\'eor\`eme 1 de \cite{Kraus2007} dans le cas o\`u $h_K^+=1$ (voir aussi la prop.   4 de {\it{loc. cit.}} pour  $d=3$). Seule la condition $h_K^+=1$ intervient ici. On n'utilise pas l'hypoth\`ese  que $2$ est totalement ramifi\'e dans $K$. 

\begin{theoreme} \label{T:th13} Soit $p$ un nombre premier ne divisant pas $D_KR_K$. Si $\rho_{E,p}$  est r\'eductible, alors $E/K$, ou bien une courbe elliptique sur $K$ li\'ee \`a $E$ par une $K$-isog\'enie de degr\'e $p$,  poss\`ede un point d'ordre $p$ rationnel sur $K$. En particulier, si $p>p_0$ alors $\rho_{E,p}$  est irr\'eductible.
\end{theoreme} 

\begin{proof} Rappelons les   principaux arguments. 
Supposons $\rho_{E,p}$ r\'eductible. Il existe des caract\`eres $\varphi, \varphi' :  \Gal(\overline{K}/K)\to \F_p^*$ tels que $\rho_{E,p}$ soit repr\'esentable sous la forme
$\begin{pmatrix} \varphi & * \\ 0 & \varphi'\ \end{pmatrix}.$

Soit $\calA_p$ l'ensemble des id\'eaux premiers de $O_K$ au-dessus de $p$. Les caract\`eres 
 $\varphi$ et  $\varphi'$ sont non ramifi\'es en tout id\'eal premier qui n'est pas dans  $\calA_p$. De plus, pour tout $\pp\in \calA_p$ l'un des caract\`eres  $\varphi$ et  $\varphi'$ est non ramifi\'e en $\pp$. 
Par suite, 
il existe un sous-ensemble $\calA$ de $\calA_p$ tel que l'un des caract\`eres 
 $\varphi$ et  $\varphi'$ soit non ramifi\'e en dehors de  $\calA$ et que pour tout $\pp\in \calA$ sa restriction \`a un  sous-groupe d'inertie en $\pp$ soit le caract\`ere cyclotomique.

Supposons $\calA$ vide. Alors, $\varphi$ ou $\varphi'$ est partout non ramifi\'e aux places finies. Parce que $h_K^+=1$, $\varphi$ ou $\varphi'$ est donc trivial. Si $\varphi=1$, $E$ 
a un point d'orde $p$ rationnel sur $K$. Si $\varphi'=1$,   $E$ est li\'ee   par  une $K$-isog\'enie de degr\'e $p$ a une courbe elliptique sur $K$ ayant un point d'ordre $p$  sur $K$. 
\vskip0pt\noindent
Si  $\calA$ n'est pas  vide, alors  $p$ divise $D_KR_K$ (voir la fin de la preuve du th. 1 de  \cite{Kraus2007}, p. 619, alin\'ea (2)), d'o\`u le r\'esultat.
\end{proof}

\begin{remark}
Si $R_K$ n'est pas nul,  on obtient ainsi une constante explicite $c_K$,  telle que pour tout $p>c_K$ et toute courbe elliptique $E/K$ semi-stable sur $K$, la repr\' esentation $\rho_{E,p}$ soit irr\'eductible. Dans ce cas, on obtient une version effective du th\'eor\`eme~\ref{T:th12}.   Par exemple, $R_K$ n'est pas nul si $d\in \left\lbrace 1,2,3,5,7\right\rbrace$ ({\it{loc. cit.}}, th. 2).
\end{remark}

\section{Le th\' eor\`eme d'abaissement du niveau}
Il s'agit de l'analoque du th\'eor\`eme de  Ribet intervenant dans la d\'emonstration du th\'eor\`eme de Fermat sur $\Q$ (\cite{Ri}). Dans notre situation, si on a $p>4d$, il s'\'enonce comme suit (\cite{FS1}, th. 7, les lemmes \ref{L:lemme4}, \ref{L:lemme5} et l'\'egalit\'e \eqref{(8.5)}) :

\medskip

\begin{theoreme} \label{T:th14} Supposons que les conditions suivantes soient satisfaites :

\begin{itemize}

\item[1)]  On a $p>4d$.

\item[2)] L'indice de ramification de tout id\'eal premier   de $O_K$ au-dessus de $p$ est strictement plus petit que $p-1$ et le corps $\Q(\mu_p)^+$ n'est pas contenu dans $K$.

\item[3)] La courbe elliptique $E_0/K$ est modulaire.

\item[4)] La repr\'esentation $\rho_{E_0,p}$ est irr\'eductible.
\end{itemize}
Alors, il  existe $\ff\in \calH$
et un id\'eal premier $\pp$ de l'anneau d'entiers $O_{\Q_{\ff}}$ de $\Q_{\ff}$ au-dessus de $p$, tels que,  en notant
$$\rho_{\ff,\pp} : \Gal(\overline{K}/K)\to \GL_2(O_{\Q_{\ff}}/\pp)$$
la repr\'esentation  galoisienne associ\'ee \`a $\ff$ et $\pp$, on ait
 \begin{equation}
 \label{(10.1)}
\rho_{E_0,p}\simeq \rho_{\ff,\pp}.
 \end{equation}
\end{theoreme} 

\medskip

\begin{proposition} \label{P:prop4}   Les hypoth\`eses faites dans l'\'enonc\'e  du th\'eor\`eme \ref{T:th14}  sont satisfaites si  $p$ est plus grand qu'une constante qui ne d\'epend que  de $K$. 
\end{proposition}

\begin{proof} 
La seconde condition est  r\'ealis\'ee si $p$ est non ramifi\'e dans $K$. Les th\'eor\`emes \ref{T:th10} et \ref{T:th12}  entra\^inent alors le r\'esultat.
\end{proof}

Les repr\'esentations $\rho_{E_0,p}$ et $\rho_{\ff,\pp}$ sont non ramifi\'ees en dehors de $\LL$ et des id\'eaux premiers de $O_K$ au-dessus de $p$.   Parce que $\rho_{E_0,p}$ est irr\'eductible, l'isomorphisme \eqref{(10.1)} se traduit par les conditions suivantes  : pour tout id\'eal premier $\fq$ de $O_K$,  distinct de $\LL$, qui n'est pas au-dessus de $p$,   on a 
\begin{equation}
 \label{(10.2)}
a_{\fq}(\ff)  \equiv a_{\fq}(E_0) \pmod \pp \quad \hbox{si}\quad E_0\ \hbox{a bonne r\'eduction en}\  \fq,
\end{equation}
\begin{equation}
 \label{(10.3)}
a_{\fq}(\ff)  \equiv \pm \left(\Norm(\fq)+1\right) \pmod \pp \quad \hbox{si}\quad  E_0\ \hbox{a  r\'eduction de type multiplicatif en}\  \fq.
\end{equation}
 
On en d\'eduit  l'\'enonc\'e ci-dessous permettant parfois d'obtenir une contradiction \`a l'existence de $(a,b,c)\in F_p(K)$ (cf. \cite{FS3}, lemme 7.1). 
Pour tout id\'eal premier $\fq$  de $O_K$, distinct de $\LL$, posons 
  \begin{equation}
   \label{(10.4)}
 A_{\fq}=\Big\lbrace t\in \Z\ \big| \ |t|\leq 2 \sqrt{\Norm(\fq)}\quad \hbox{et}\quad \ \Norm(\fq)+1\equiv t \pmod 4 \Big\rbrace,
  \end{equation}
    \begin{equation}
     \label{(10.5)}
B_{\ff,\fq}=\Norm(\fq)  \Bigl(\bigl(\Norm(\fq)+1\bigr)^2-a_{\fq}(\ff)^2\Bigr) \prod_{t\in A_{\fq}} \bigl(t-a_{\fq}(\ff)\bigr).
 \end{equation}
 
 \begin{proposition}  \label{P:prop5}  Supposons les quatre conditions du th\'eor\`eme~\ref{T:th14} satisfaites. Soient $\ff\in \calH$ et $\pp$ un id\'eal premier  de  $O_{\Q_{\ff}}$   au-dessus de $p$ tels que  $\rho_{E_0,p}\simeq \rho_{\ff,\pp}$. 
 Soit $\fq$ un id\'eal premier de $O_K$ distinct de $\LL$. Alors, $p$ divise la norme de $\Q_{\ff}$ sur $\Q$ de $B_{\ff,\fq}$.
\end{proposition}

\begin{proof}  Si $\fq$ divise $p$, alors $p$ divise $\Norm(\fq)$, en particulier $p$ divise la norme de $\Q_{\ff}$ sur $\Q$ de $B_{\ff,\fq}$. Supposons que $\fq$ ne divise pas $p$. La courbe elliptique  $E_0$ a bonne r\'eduction ou r\'eduction de type multiplicatif en $\fq$. 

Supposons que $E_0$ ait bonne r\'eduction en $\fq$. Parce que $E_0$ a tous ses points d'ordre $2$ rationnels que $K$ et que $\fq$ est distinct de $\LL$, le nombre de points de la courbe elliptique d\'eduite de $E_0$ par r\'eduction est multiple de $4$. Par ailleurs, on a $|a_{\fq}(E_0)|\leq 2 \sqrt{\Norm(\fq)}$ (borne de Weil), donc 
$a_{\fq}(E_0)$ appartient \`a $A_{\fq}$. La condition \eqref{(10.2)} implique alors notre assertion dans ce cas.

Si $E_0$ a   r\'eduction de type multiplicatif en $\fq$, la condition \eqref{(10.3)} est satisfaite, d'o\`u le r\'esultat.
\end{proof}

\section{D\'emonstration du th\'eor\`eme \ref{T:th4}}
Compte tenu des propositions \ref{P:prop4} et \ref{P:prop5}, le th\'eor\`eme \ref{T:th4} r\'esulte   de l'\'enonc\'e suivant :

\begin{proposition} \label{P:prop6} Pour tout $\ff\in \calH$, il existe  id\'eal premier  $\fq$  de $O_K$, distinct de $\LL$, tel que l'on ait  $B_{\ff,\fq}\neq 0$.
\end{proposition}

\begin{proof}
Soit $\ff$ un \'el\'ement de $\calH$.

Supposons $\Q_{\ff}\neq \Q$. Il existe alors un id\'eal premier  $\fq$ de $O_K$, distinct de $\LL$,  tel que $a_{\fq}(\ff)$ ne soit pas dans $\Z$ (cf. \cite{Dembelecremona}, th. 9), d'o\`u $B_{\ff,\fq}\neq 0$.

 Supposons $\Q_{\ff}=\Q$. D'apr\`es la condition ${\bf{(C)}}$, il existe un id\'eal premier  $\fq$ de $O_K$, distinct de $ \LL$,   tel que l'on ait
 $a_{\fq}(\ff)\not\equiv \Norm(\fq)+1 \pmod 4$. En particulier,   $a_{\fq}(\ff)$ n'est  pas dans $A_{\fq}$. De plus, on a   $\bigl(\Norm(\fq)+1\bigr)^2\neq a_{\fq}(\ff)^2$ : dans le cas contraire, on aurait $a_{\fq}(\ff)=-(\Norm(\fq)+1)$. Or $\fq$ \'etant distinct de $\LL$, on a $2(\Norm(\fq)+1)\equiv 0 \pmod 4$, ce qui conduit \`a une contradiction.
Par suite, on a $B_{\ff,\fq}\neq 0$, d'o\`u l'assertion.
\end{proof}

\bigskip

{\large{{\bf{Partie 4. Les th\'eor\`emes 6, 7, 8  et 9}}}}

Dans toute cette partie, on suppose qu'il existe un point $(a,b,c)\in F_p(K)$ tel que $abc\neq 0$. Rappelons que pour $p>4d$, on suppose implicitement  qu'il est  normalis\'e  comme indiqu\' e dans l'\'enonc\'e de la proposition \ref{P:prop3}.

\section{Sur l'irr\'eductibilit\'e de $\rho_{E_0,p}$} 

Dans le cas o\`u $p$ est ramifi\'e dans $K$, le th\'eor\`eme~\ref{T:th13} ne permet pas  d'\'etablir que la repr\'esentation  $\rho_{E_0,p}$ est irr\'eductible (si tel est le cas). 
On dispose n\'eanmoins du r\'esultat suivant permettant parfois de conclure, qui vaut  sans hypoth\`ese de ramification en $p$. Pour tout cycle $\mm$ of $K$, notons $K^{\mm}$ le corps de classes de rayon  modulo  $\mm$ sur $K$.

\begin{lemme} \label{L:lemme6} Soit  $\pp$ un id\'eal premier de $O_K$ au-dessus de  $p$.  Notons $\mm_{\infty}$ le produit des places archim\'ediennes de $K$.
Soit 
 $\varphi : \GalK \to \F_p^*$ un caract\`ere non ramifi\'e  en dehors  $\mm_{\infty}\pp$.  Alors,  le corps laiss\'e fixe par le noyau de  $\varphi$ est contenu
dans $K^{\mm_{\infty}\pp}$.
\end{lemme}

\begin{proof} Soit  $n\geq 1$ un entier.  Il suffit de montrer que 
\begin{equation}
\label{(12.1)}
\Gal\bigl(K^{\mm_{\infty} \pp^{n}}/K^{\mm_{\infty} \pp}\bigr) \quad \hbox{est un }\ p\hbox{-groupe}.
\end{equation}
En effet, d'apr\`es l'hypoth\`ese faite,  il existe   $j\geq 1$ tel que le  corps laiss\'e fixe par le noyau de  $\varphi$ soit contenu dans 
$K^{\mm_{\infty} \pp^j}$. D'apr\`es l'assertion \eqref{(12.1)}, le groupe  $\Gal\bigl(K^{\mm_{\infty} \pp^j}/K^{\mm_{\infty} \pp}\bigr)$
est contenu dans le noyau  $\varphi$, ce qui implique alors le r\'esultat.

D\'emontrons    \eqref{(12.1)}. Posons 
$\mm=m_{\infty}\pp^{n+1}$ et   $\nn=\mm_{\infty}\pp^n$. Notons
$U_{\mm,1}$  le groupe des unit\'es de $O_K$ congrues \`a $1$ modulo $\mm$ et 
$U_{\nn,1}$ l'analogue de $U_{\mm,1}$ en ce qui concerne le cycle  $\nn$. Le corollaire  3.2.4 of \cite{Cohen} entra\^ine l'\' egalit\'e 
\begin{equation}
\label{(12.2)}
[K^{\mm}:K^{\nn}](U_{\nn,1} :U_{\mm,1})=\Norm(\pp).
\end{equation}
Par ailleurs, pour tout $x\in U_{\nn,1}$, on a  $x^p\in U_{\mm,1}$. 
 Ainsi, $U_{\nn,1}/U_{\mm,1}$ est un  $p$-groupe. D'apr\`es  l'\'egalit\'e  \eqref{(12.2)}, 
$[K^{\mm}:K^{\nn}]$  est donc une puissance de  $p$, ce qui entra\^ine l'assertion \eqref{(12.1)}.
\end{proof}

 On utilisera   ce r\'esultat de la fa\c con suivante. Supposons $p>4d$ et $\rho_{E_0,p}$  r\'eductible. Soient $\varphi$ et $\varphi'$ ses caract\`eres d'isog\'enie. Ils sont non ramifi\'es en dehors de $\mm_{\infty}$ et des id\'eaux premiers de $O_K$ au-dessus de $p$. Supposons  qu'il existe un id\'eal premier $\pp$ de $O_K$ au-dessus de $p$ tel que 
 $\varphi$ ou  $\varphi'$   soit non ramifi\'e en dehors de  $\mm_{\infty}\pp$  et  que de plus on ait $[K^{\mm_{\infty}\pp}:K]\leq 2$. On d\'eduit alors  du lemme~\ref{L:lemme6},  l'existence d'une courbe elliptique sur $K$ ayant un point d'ordre $p$ rationnel sur $K$, ce qui,  si $p$ est assez grand par rapport \`a $d$,  conduit   \`a une contradiction.

\section{Corps cubiques et modularit\'e}

On  va d\'emontrer  ici un crit\`ere permettant parfois de d'\'etablir  que toute courbe elliptique semi-stable d\'efinie sur un corps cubique r\'eel  est modulaire.
Rappelons que l'entier $R_K$ est d\'efini par la seconde formule de \eqref{(9.4)}. 

\begin{theoreme} \label{T:th15} Soit $K$ un corps cubique  r\'eel satisfaisant les   conditions suivantes :

\begin{itemize}

\item[1)]  On a $h_K^+=1$.

\item[2)] $5$ et $7$ ne divisent pas $D_KR_K$. 

\item[3)] $3$ n'est pas inerte dans $K$.

\end{itemize}

Alors, toute courbe elliptique semi-stable d\'efinie sur $K$ est modulaire.

\end{theoreme}

\subsection{Courbes elliptiques et points de $35$-torsion}
Commen\c cons par \'etablir  l'\'enonc\'e  qui suit, qui est une cons\'equence d'un r\'esultat de Bruin et Najman (\cite{BruiniNajman}).

\begin{proposition}  \label{P:prop7} Soit $K$ un corps cubique tel que $3$ ne soit pas inerte dans $K$. Alors, il n'existe pas  de courbes elliptiques  d\'efinies sur $K$ ayant  un point d'ordre $35$ rationnel sur $K$.
\end{proposition} 

\begin{proof} On utilise le th\'eor\`eme 1 de \cite{BruiniNajman}, ainsi que  la remarque 4  \`a la fin du paragraphe 2  de {\it{loc. cit.}}  qui est tr\`es utile dans son 
application.
Avec  les notations de ce th\'eor\`eme, on prend
$$A=\Z/35\Z,\quad L=\Q,\quad m=1,    \quad n=35,\quad X=X'=X_1(35),\quad  \pi=\hbox{id} \quad \hbox{et}\quad p=\pp_0=3.$$
Il s'agit de v\'erifier que les six conditions i)-vi) de cet \'enonc\' e sont satisfaites.
Parce que $3$ ne divise pas $n$, on a $A'=\Z/35\Z$ et $h=1$. Pour toute pointe $Z$ de $X_1(35)$, l'ensemble $L(Z)$ est le corps de rationalit\'e   de $Z$. C'est donc l'un des corps 
$$\Q,\quad \Q(\mu_5),\quad \Q(\mu_7)\quad \hbox{et}\quad \Q(\mu_{35})^+.$$
Par hypoth\`ese,  $3$ n'est pas inerte dans $K$, on a donc
$$S_{K,\pp_0}=\left\lbrace 1,2\right\rbrace.$$
La gonalit\'e de $X_1(35)$ vaut $12$  et sa Jacobienne est de rang $0$ sur $\Q$ (\cite{Derickx}, p. 19 et lemme 1, p. 30). Pour tout id\'eal premier $\pp$ de $O_K$ au-dessus de $3$, 
il n'existe pas de courbes elliptiques d\'efinies sur $k(\pp)$ ayant un point rationnel d'ordre $35$.  Par ailleurs, $3$ est inerte dans 
$\Q(\mu_5)$, $ \Q(\mu_7)$ et $\Q(\mu_{35})^+.$ Les six conditions consid\'er\'ees sont donc satisfaites, d'o\`u le r\'esultat.
\end{proof}

\begin{remark}
Il existe des courbes elliptiques d\'efinies sur $\F_{27}$ ayant un point rationnel d'ordre $35$, ce qui explique l'hypoth\`ese que $3$ n'est pas inerte dans $K$ dans l'\'enonc\'e de la proposition (cf. \cite{Waterhouse}, th. 4.1).
 \end{remark}

\subsection{D\'emonstration du th\'eor\`eme 15}  Soit $E/K$ une courbe elliptique semi-stable sur $K$. On utilise le th\'eor\`eme~\ref{T:th11}. Par hypoth\`ese, 
 $5$ et $7$ sont  non ramifi\'es dans $K$. Il s'agit ainsi de montrer que l'une au moins des repr\'esentations $\rho_{E,5}$ et $\rho_{E,7}$ est irr\'eductible. 
Supposons le contraire i.e. que $\rho_{E,5}$ et $\rho_{E,7}$ soient r\'eductibles. 
Parce que l'on a $h_K^+=1$ et que $5$ et $7$ ne divisent  pas $D_KR_K$, quitte \`a remplacer $E$ par une courbe elliptique sur $K$ qui lui est li\'ee   par une $K$-isog\'enie de degr\'e $1$, $5$, $7$ ou  $35$,  on peut supposer que $E$ a un point d'ordre $5$ et un point d'ordre $7$ rationnels sur $K$ (th.~\ref{T:th13}). Elle poss\`ede donc un point d'ordre $35$ rationnel sur $K$, ce qui conduit \`a une contradiction (prop.~\ref{P:prop7}),  d'o\`u  le r\'esultat.

\section{Corps cubiques et irr\'eductibilit\'e de $\rho_{E_0,13}$}

On utilisera dans la d\'emonstration du th\'eor\`eme~\ref{T:th6}  le r\'esultat suivant.

\begin{theoreme} \label{T:th16} Soit $K$ un corps cubique  satisfaisant les   conditions suivantes :

\begin{itemize}

\item[1)] On a $h_K^+=1$.

\item[2)]  $13$ ne divise pas $D_KR_K$. 

\item[3)] $3$ n'est pas inerte dans $K$. 

\end{itemize}

Alors, pour toute courbe elliptique semi-stable $E/K$, ayant un point d'ordre $2$ rationnel sur $K$, la repr\'esentation $\rho_{E,13}$ est irr\'eductible.
\end{theoreme}

\begin{proof} Elle est analogue \`a celle du th\' eor\`eme~\ref{T:th15}. Soit $E/K$ une courbe elliptique semi-stable ayant un point d'ordre $2$ rationnel sur $K$. 
Supposons  $\rho_{E,13}$ r\'eductible. Parce que $h_K^+=1$ et que $13$ ne divise pas $D_KR_K$, la courbe elliptique $E$, ou bien une courbe elliptique sur $K$ li\'ee \`a $E$ par une $K$-isog\' enie de degr\' e $13$, poss\`ede un point d'ordre $13$ rationnel $K$ (th. \ref{T:th13}). Il  existe donc  une courbe elliptique sur $K$ ayant  un point d'ordre $26$ rationnel sur $K$. 

Avec les notations du  th\'eor\`eme 1 de \cite{BruiniNajman}, on prend $A=\Z/26\Z$, 
$m=1$, $n=26$, $\pp_0=3$, 
$X=X'=X_1(26)$ et $\pi$ est l'identit\'e de $X$. Parce que $3$ ne divise pas $n$, on a $A'=\Z/26\Z$ et $h=1$. Le corps de rationalit\'e des pointes de $X_1(26)$ est $\Q$ ou $\Q(\mu_{13})^+$. 
Par hypoth\`ese,    $3$ n'est pas inerte dans $K$, donc on a  
$S_{K,\pp_0}=\left\lbrace 1,2\right\rbrace$. La gonalit\'e de $X_1(26)$ vaut $6$ et sa Jacobienne est de rang $0$ sur $\Q$ (\cite{Derickx}, p. 19 et lemme 1, p. 30). Pour tout id\'eal premier $\pp$ de $O_K$ au-dessus de $3$, 
il n'existe pas de courbes elliptiques d\'efinies sur $k(\pp)$ ayant un point rationnel d'ordre $26$.  Par ailleurs, dans l'anneau d'entiers de 
 $\Q(\mu_{13})^+$, l'id\'eal engendr\'e par $3$ est le produit de deux  id\'eaux premiers de degr\'e 3, et $3$ n'est pas dans $S_{K,\pp_0}$.   Le th\'eor\`eme 1 de \cite{BruiniNajman} entra\^ine  alors contradiction et le r\'esultat.
\end{proof}

\section{D\'emonstration du  th\'eor\`eme~\ref{T:th6}}

Gross et Rohrlich ont d\'emontr\'e que l'ensemble des points rationnels de  $F_7$ et $F_{11}$ 
sur tout corps cubique est   trivial (\cite{GR}, th. 5.1). Par ailleurs, Klassen et Tzermias ont \'etabli
qu'il en est de m\^eme pour $F_5$ (\cite{KlassenTzermias}, th. 1).
On supposera donc que l'on a
$$p\geq 13.$$
En particulier, l'in\'egalit\'e $p>4d$ est satisfaite. De plus,  on a $p_0=13$  (\cite{Parent}).

\subsection{Cas o\`u $D_K=148$}

On a (\cite{Voight})
$$K=\Q(\alpha)\quad \hbox{o\`u}\quad \alpha^3-\alpha^2-3\alpha+1=0.$$
Le nombre premier  $3$ \'etant inerte dans $K$,   les th\'eor\`emes~\ref{T:th15} et~\ref{T:th16} ne s'appliquent  pas.

\begin{lemme}  \label{L:lemme7}  La courbe elliptique $E_0/K$ est modulaire. 
\end{lemme} 

\begin{proof}
On utilise   le th\'eor\`eme 11 avec $\ell=5$, qui ne divise pas $D_K$.  Supposons que $E_0$ poss\`ede un sous-groupe d'ordre $5$ stable par $\Gal(\overline{K}/K)$.   Parce que $E_0$ a  tous ses points d'ordre $2$ rationnels sur $K$, il en r\'esulte que $E_0$ est li\'ee par une $K$-isog\'enie de degr\'e au plus $2$ \`a une courbe elliptique sur $K$ ayant un sous-groupe cyclique d'ordre $20$ stable par $\Gal(\overline{K}/K)$ (cf. par exemple  \cite{AnniSiksek}, p. 1163). 
La courbe modulaire $X_0(20)$  est la courbe elliptique, de conducteur $20$, num\'erot\'ee 20A1 dans les tables de Cremona,  d'\'equation  (\cite{Cremona})
$$y^2=x^3+x^2+4x+4.$$
 Elle poss\`ede six pointes,  toutes rationnelles sur $\Q$. Le groupe $X_0(20)(\Q)$ est d'ordre $6$, et on v\'erifie avec {\tt Magma} qu'il en est de m\^eme du groupe  $X_0(20)(K)$. Cela montre que $Y_0(20)(K)$ est vide, d'o\`u une contradiction et notre assertion.
\end{proof}

\begin{lemme} \label{L:lemme8}   La repr\'esentation $\rho_{E_0,p}$ est irr\'eductible. 
\end{lemme} 

\begin{proof}
Posons  $u=\alpha^2$.  C'est une unit\'e totalement positive de $K$. Son polyn\^ome minimal est $H=X^3-7X^2+11X-1$. On a $H(1)=4$, donc $R_K$ divise $4$ (formules \eqref{(9.2)} et \eqref{(9.4)}). Par ailleurs, on a $D_K=4.37$. Cela entra\^ine le r\'esultat  si $p\neq 13, 37$ (th.~\ref{T:th13}).

Supposons $\rho_{E_0,13}$  r\'eductible.  Dans ce cas, $E_0$ est li\'ee par une $K$-isog\'enie de degr\'e au plus $2$ \`a une courbe elliptique sur $K$ ayant un sous-groupe cyclique d'ordre $52$ stable par $\Gal(\overline{K}/K)$.
Il existe un morphisme d\'efini sur $\Q$, de degr\'e  $3$, de la courbe modulaire  $X_0(52)$  sur  la courbe elliptique $F/\Q$, num\'erot\'ee 52A1 dans les tables de Cremona,  d'\'equation  (\cite{Cremona}, p. 363)
$$y^2=x^3+x-10.$$
La courbe $X_0(52)$ poss\`ede six pointes, toutes rationnelles sur $\Q$. Avec {\tt Magma}, on constate   que l'on a $F(K)=F(\Q)$, qui est d'ordre $2$.    On en d\'eduit que $Y_0(52)(K)$ est vide, d'o\`u une contradiction et le fait que  $\rho_{E_0,13}$ soit  irr\'eductible. 

Supposons $\rho_{E_0,37}$  r\'eductible. Soient $\varphi$ et $\varphi'$ ses caract\`eres d'isog\'enie. On a $37O_K=\pp_1^2\pp_2$,  o\`u $\pp_i$ est un id\'eal premier de $O_K$.  L'id\'eal  $\pp_2$  est non ramifi\'e. D'apr\`es l'hypoth\`ese faite sur $\rho_{E_0,37}$, si  $E_0$ a  bonne r\'eduction en $\pp_2$,  cette r\'eduction est n\' ecessairement de hauteur $1$ (cf. \cite{Se1}, prop. 12). On en d\'eduit que 
 l'un des caract\`eres $\varphi$ et $\varphi'$ est non ramifi\'e en $\pp_2$ ({\it{loc. cit.}}, cor. p. 274 et cor. p. 277). Quitte \`a remplacer $E_0$ par une courbe elliptique qui lui est li\'e par une $K$-isog\'enie de  degr\'e 37, on peut supposer que c'est $\varphi$. Par suite, $\varphi$ est non ramifi\'e en dehors  de $\pp_1$
et des places archim\'ediennes.
D'apr\`es le lemme \ref{L:lemme6}, le corps laiss\'e fixe par le noyau de $\varphi$ est donc contenu dans le corps de rayon $K^{\mm_{\infty} \pp_1}$. On v\'erifie  que l'on a 
$[K^{\mm_{\infty} \pp_1}:K]=2$ (\cite{Pari}). Ainsi, $\varphi$ est d'ordre au plus $2$. On a $\varphi\neq 1$, car $E_0$ n'a  pas de   point d'ordre $37$ rationnel sur $K$. Le caract\`ere  $\varphi$ est donc d'ordre $2$ et la courbe elliptique d\'eduite de $E_0$ par torsion quadratique  par $\varphi$ a donc un point d'ordre $37$ sur $K$, d'o\`u une contradiction et le r\'esultat.
\end{proof}

Les quatre    conditions du th\'eor\`eme d'abaissement de niveau sont  donc satisfaites. Par ailleurs, on a $|\calH|=0$ i.e. il n'existe pas de newforms modulaires paraboliques de Hilbert sur $K$ de poids parall\`ele $2$ et de niveau $\LL$ (\cite{lmfdb}). On obtient ainsi une contradiction \`a l'existence de $(a,b,c)$, d'o\`u le th\'eor\`eme  dans ce cas.

\subsection{Cas o\`u $D_K=404$}
On a  
$$K=\Q(\alpha)\quad \hbox{o\`u}\quad \alpha^3-\alpha^2-5\alpha-1=0.$$
On a  $3O_K=\wp_1\wp_2$, o\`u $\wp_1$ est un id\'eal premier de  degr\'e 1 et o\`u $\wp_2$ est de degr\'e 2. En particulier, $3$ n'est pas inerte dans $K$.

Le polyn\^ome minimal de $\alpha^2\in U_K^+$ est $H=X^3-11X^2+23X-1$ et on a $H(1)=12$. On a $D_K=2^2.101$, donc $5,7$ et $13$ ne divisent pas $D_KR_K$. 

Il r\'esulte alors du th\'eor\`eme~\ref{T:th15} que $E_0/K$ est modulaire. Pour $p\neq 101$, les  th\'eor\`emes~\ref{T:th13} et~\ref{T:th16} 
entra\^inent que $\rho_{E_0,p}$ est irrr\'eductible.  La d\'ecomposition de  $101O_K$ en produit d'id\'eaux premiers est de la forme $\pp_1^2\pp_2$ et on a $[K^{\mm_{\infty}\pp_1}:K]=2$. On en d\'eduit,  comme dans la d\'emonstration du lemme~\ref{L:lemme8},  que  $\rho_{E_0,101}$ est irrr\'eductible.

Par ailleurs, on a $|\calH|=1$ i.e. il existe une unique  newform modulaire parabolique  de Hilbert $\ff$ sur $K$ de poids parall\`ele $2$ et de niveau $\LL$ (\cite{lmfdb}). En particulier, on a $\Q_{\ff}=\Q$. 
Soit $\fq$ l'id\'eal premier de $O_K$ au-dessus de $7$ de degr\'e 1. On a $a_{\fq}(\ff)=-2$. D'apr\`es les \'egalit\'es \eqref{(10.4)} et \eqref{(10.5)}, on a
$$A_{\fq}=\left\lbrace-4,0,4 \right\rbrace\quad \hbox{et}\quad B_{\ff,\fq}=-2^5.3^2.5.7.$$
Ainsi, $B_{\ff,\fq}$ n'est pas divisible par $p$, d'o\`u la conclusion dans ce cas (prop.~\ref{P:prop5}). 

\subsection{Cas o\`u $D_K=564$}
On a  
$$K=\Q(\alpha)\quad \hbox{o\`u}\quad \alpha^3-\alpha^2-5\alpha+3=0.$$
On a  $3O_K=\wp_1^2\wp_2$, o\`u $\wp_i$ est un id\'eal premier de  degr\'e 1.

Le polyn\^ome minimal  de $(\alpha+2)^2\in U_K^+$  est $H=X^3-27X^2+135X-1$ et on a $H(1)=2^2.3^3$. 

On a $D_K=2^2.3.47$. On en d\'eduit  que $E_0/K$ est modulaire (th.~\ref{T:th15}) et que $\rho_{E_0,p}$ est irr\'eductible pour $p\neq 47$ (th. \ref{T:th13} et 
th. \ref{T:th16}). On  $47O_K=\pp_1^2\pp_2$ et  $[K^{\mm_{\infty}\pp_1}:K]=2$, il en est donc de m\^eme de $\rho_{E_0,47}$.

L'ensemble $\calH$, qui est de cardinal $2$,  est form\'e d'une newform $\ff$ et de sa conjugu\'ee galoisienne telle que  $\Q_{\ff}=\Q(\beta)$ o\`u $\beta^2+3\beta-1=0$ (\cite{lmfdb}). Soit $\fq$ l'id\'eal premier de $O_K$ au-dessus de $3$  tel que
$a_{\fq}(\ff)=\beta$. Il est de degr\'e $1$. On a $A_{\fq}=\left\lbrace 0 \right\rbrace$ et $B_{\ff,\fq}=-3\beta (16-\beta^2)$. Sa norme sur $\Q$ \' etant $-3^6$, on obtient  la conclusion cherch\'ee.

Cela termine la d\'emonstration du th\'eor\`eme~\ref{T:th6}.

\section{D\'emonstration du   th\'eor\`eme~\ref{T:th7}} 

L'ensemble des points rationnels de  $F_{11}$ 
sur tout corps de degr\'e 5 sur $\Q$ est  trivial (\cite{GR}, th. 5.1).  Tzermias a d\'emontr\'e 
qu'il en est de m\^eme pour $F_7$ (\cite{Tzermias}, th. 1).
On supposera donc que l'on a
$$p\geq 23.$$
En particulier, on a $p>4d$. Par ailleurs, on a $D_K=2^4.7877$. 

\begin{lemme}  \label{L:lemme9}  La courbe elliptique $E_0/K$ est modulaire. 
\end{lemme} 

\begin{proof}
On utilise   le th\'eor\`eme~\ref{T:th11} avec $\ell=7$. Supposons que $E_0$ poss\`ede un sous-groupe d'ordre $7$ stable par $\Gal(\overline{K}/K)$.   Dans ce cas, $E_0$ poss\`ede un sous-groupe cyclique d'ordre $14$ stable par $\Gal(\overline{K}/K)$.
La courbe modulaire $X_0(14)$  est la courbe elliptique, de conducteur $14$, num\'erot\'ee 14A1 dans les tables de Cremona,  d'\'equation  (\cite{Cremona})
$$y^2+xy+y=x^3+4x-6.$$
 Elle poss\`ede quatre  pointes,  qui sont  rationnelles sur $\Q$. On v\'erifie avec {\tt Magma} que  l'on a $X_0(14)(K)=X_0(14)(\Q)$ qui est d'ordre $6$.  Par ailleurs, les points non cuspidaux de $X_0(14)$ correspondent \`a deux classes d'isomorphisme de courbes elliptiques sur $\Q$ d'invariants modulaires  entiers ($-15^3$ et $255^3$). Parce que celui de $E_0$ n'est pas entier en $\LL$, on obtient une contradiction et le r\'esultat.
  \end{proof}

\begin{lemme}  \label{L:lemme10} La repr\'esentation $\rho_{E_0,p}$ est irr\'eductible. 
\end{lemme} 

\begin{proof} L'entier $\alpha$  \'etant d\'efini par l'\'egalite \eqref{(5.1)}, 
posons 
$$u_1=(\alpha-1)^2\quad \hbox{et}\quad u_2=(\alpha^2+\alpha-1)^2.$$ Ce sont des unit\'es totalement positives de $O_K$.  On v\'erifie que l'on a (formules \eqref{(9.2)} et \eqref{(9.3)})
$$H_1^{(u_1)}(1)=-12\quad \hbox{et}\quad \pgcd\left(H_2^{(u_1)}(1),H_2^{(u_2)}(1)\right)=2^{12}.3.5^2.$$
Il en r\'esulte que $R_K$ n'est pas divisible pas un nombre premier plus grand que $7$. On a $p_0=19$ (\cite{Derickx},  chap. III, th. 1.1),  d'o\`u l'assertion si $p\neq 7877$. Par ailleurs, on a $7877O_K=\pp_1^2\pp_2$, o\`u $\pp_1$ est un id\'eal premier de degr\'e $1$ et $\pp_2$  un id\'eal premier de degr\'e $3$. On v\'erifie     que l'on a  
$[K^{\mm_{\infty} \pp_1}:K]=2$, ce qui entra\^ine le r\'esultat pour $p=7877$.
\end{proof}

Par ailleurs, on a $|\calH|=2$. Plus pr\'ecis\'ement, $\calH$ est form\'e d'une newform $\ff$ et de sa conjugu\'ee galoisienne, et on a $\Q_{\ff}=\Q(\beta)$ o\`u $\beta^2+\beta-3=0$ (\cite{lmfdb}). Soit $\fq$ l'id\'eal premier de $O_K$ au-dessus de $3$  de degr\'e $1$. On a 
$a_{\fq}(\ff)=\beta$. D'apr\`es les \'egalit\'es \eqref{(10.4)} et \eqref{(10.5)}, on a
$A_{\fq}=\left\lbrace 0 \right\rbrace$ et $B_{\ff,\fq}=-3\beta(16-\beta^2).$
La norme sur $\Q$ de $B_{\ff,\fq}$ est $-3^5.17$, qui n'est pas divisible par $p$. La proposition~\ref{P:prop5} implique alors le  th\'eor\`eme.

\section{D\'emonstration du  th\'eor\`eme~\ref{T:th8}} 

On a $D_K=2^{11}.37^2$. On v\'erifie, comme dans la d\'emonstration du lemme~\ref{L:lemme9}, que l'on a $X_0(14)(\Q)=X_0(14)(K)$, d'o\`u l'on d\'eduit que $E_0/K$ est modulaire. 

D\'emontrons que $\rho_{E_0,p}$ est irr\'eductible. Par hypoth\`ese, on a $p>4d=24$ et $p\neq 37$. L'entier $\alpha$ \'etant d\'efini par l'\'egalit\'e \eqref{(5.2)}, on consid\`ere les unit\'es totalement positives
$$u_1=\frac{1}{58^2}\left(4\alpha^5 + 19\alpha^4 - 28\alpha^3 - 170\alpha^2 - 16\alpha + 41\right)^2,$$
$$u_2=\frac{1}{58^2} \left(14\alpha^5 + 23\alpha^4 - 156\alpha^3 - 160\alpha^2 + 176\alpha + 13\right)^2.$$
On v\'erifie que l'on a 
$$H_1^{(u_1)}(1)=16,\quad H_2^{(u_1)}(1)=2^{32}.5^4\quad \hbox{et}\quad   H_3^{(u_2)}(1)=2^{216}.7^{54}.$$
Par suite, $R_K$ n'est pas divisible par un nombre premier plus grand que $7$.
Il  n'existe pas de courbes elliptiques sur $K$ ayant un point d'ordre $p$ rationnel sur $K$ (\cite{Derickx}, chap. III, th. 1.1)), d'o\`u l'assertion.

Par ailleurs, on v\'erifie avec  {\tt Magma} que l'on a  $|\calH|=2$ et  que $\calH$ est form\'e d'une newform $\ff$ et de sa conjugu\'ee galoisienne, dont le corps de rationalit\'e est  $\Q_{\ff}=\Q(\beta)$ o\`u $\beta^2-\beta-21=0$. 
En utilisant la proposition~\ref{P:prop5}, on  conclut alors en consid\'erant  l'id\'eal premier $\fq_1$ de $O_K$ de degr\'e $1$ au-dessus de $17$,  tel que $a_{\fq_1}(\ff)=\beta$  et  l'id\'eal premier $\fq_2$ de degr\'e $1$ au-dessus de $23$,   tel que $a_{\fq_2}(\ff)=\beta-2$.

\section{D\'emonstration du  th\'eor\`eme~\ref{T:th9}} 

Posons $K_n=\Q\left(\mu_{2^{n+2}}\right)^+$.

\begin{lemme}   \label{L:lemme11} Toute  courbe elliptique d\'efinie sur $K_n$ est modulaire. 
\end{lemme} 

\begin{proof}
Le corps $K_n$ est le $n$-i\`eme \'etage de la $\Z_2$-extension cyclotomique de $\Q$, d'o\`u l'assertion (\cite{Thorne}, th. 1).
\end{proof}

 \subsection{L'assertion 1} Elle est d\'ej\`a  connue si $p=7$ (\cite{Tzermias}, th. 1) et si   $p=11$ (\cite{GR}). Le fait que $F_5(K_2)$ soit trivial est une cons\'equence directe du th\'eor\`eme 2 de 
\cite{Kraus2017}.   On supposera donc 
$$p\geq 13.$$

On est  amen\'e \`a distinguer deux cas suivant que $p=13$  ou $p\geq 17$ (car $4d=16$).

 \subsubsection{Cas o\`u $p\geq 17$} 

\begin{lemme}  \label{L:lemme12}  Pour tout $p\geq 17$,  la repr\'esentation $\rho_{E_0,p}$ est irr\'eductible.
\end{lemme} 

\begin{proof} On a $D_{K_2}=2^{11}$ et  
$K_2=\Q(\alpha)$ o\`u 
$$\alpha^4-4\alpha^2+2=0.$$
 Posons $u=(\alpha+1)^2$. On a $u\in U_{K_2}^+$. Son polyn\^ome minimal est 
$H_1^{(u)}=X^4-12X^3+34X^2-20X+1$ et on a $H_1^{(u)}(1)=4$. On v\'erifie que $H_2^{(u)}(1)=2^{16}.17$. Par ailleurs, on a  $p_0=17$ (\cite{Derickx}, th. 1.1), d'o\`u le r\'esultat si $p\neq 17$ (th.~\ref{T:th13}). 

 Supposons $p=17$.  La courbe modulaire $X_0(17)$   est $\Q$-isomorphe \`a la courbe elliptique de conducteur $17$, not\'ee $17A1$ dans les tables de Cremona,  d'\'equation
$$y^2+xy+y=x^3-x^2-x-14.$$
Elle poss\`ede deux pointes, qui sont rationnelles sur $\Q$. 
Avec {\tt Magma},  on constate que l'on a $X_0(17)(K)=X_0(17)(\Q)$, qui est  cyclique d'ordre $4$. Les points non cuspidaux de $X_0(17)(\Q)$ correspondent \`a deux classes d'isomorphisme de courbes elliptiques sur $\Q$ d'invariants modulaires (voir par exemple \cite{Dahmen}, p. 30, Table 2.1)
$$j_1=-{17.373^3\over 2^{17}}\quad \hbox{et}\quad j_2=-{17^2.101^3\over 2}.$$
Ils  sont distincts de $j(E_0)$. En effet, on a 
$$j(E_0)=2^8\biggl({(a^{34}+(ab)^{17}+b^{34})^3\over (abc)^{34}}\biggr),$$
d'o\`u $v_{\LL}(j(E_0))=32-34v_{\LL}(abc).$ On a $v_{\LL}(j_1)=-68$ et $v_{\LL}(j_2)=-4$, ce qui entra\^ine  l'assertion et  le lemme.
\end{proof}

 On constate dans les tables de  \cite{lmfdb} que l'on a $|\calH|=0$, d'o\`u le r\'esultat dans ce cas.

 \subsubsection{Cas o\`u $p=13$} 
 
A priori, on ne peut plus  normaliser $(a,b,c)\in F_{13}(K_2)$ de sorte que la courbe de Frey $E_0/K_2$ soit semi-stable. On va donc utiliser   le th\'eor\`eme d'abaissement du niveau dans le cas g\'en\'eral (\cite{FS1}, th. 7). 

Consid\'erons  un point  $(a,b,c)\in F_{13}(K_2)$ tel que $abc\neq 0$ et que $a,b,c$ soient premiers entre eux dans $O_{K_2}$. On a $v_{\LL}(abc)\geq 1$. Soit   $E_0/K_2$ la courbe elliptique  d'\'equation 
\begin{equation}
\label{(18.1)}
 y^2=x(x-a^{13})(x+b^{13}).
  \end{equation}
 Rappelons que l'on a
 \begin{equation}
 \label{(18.2)}
 c_4(E_0)=2^4\left(a^{26}+(ab)^{13}+b^{26}\right) \quad \hbox{et}\quad \Delta(E_0)=2^4(abc)^{26}.
 \end{equation}
 
 \begin{lemme}  \label{L:lemme13}  La  repr\'esentation $\rho_{E_0,13}$ est irr\'eductible.
\end{lemme} 
 
 \begin{proof} Elle est identique \`a celle du lemme~\ref{L:lemme8}  dans le cas o\`u $p=13$.  En effet,
$F/\Q$ \' etant  la courbe elliptique   num\'erot\'ee 52A1 dans les tables de Cremona, on a $F(K_2)=F(\Q)$, qui est d'ordre $2$. Cela entra\^ine que $Y_0(52)(K_2)$ est vide et le r\'esultat.
 \end{proof}
 
 Supposons $v_{\LL}(abc)\geq 2$. Dans ce cas, par les m\^emes arguments que ceux utilis\'es dans la d\'emonstration de la proposition \ref{P:prop3}, on peut encore normaliser $(a,b,c)$ de sorte que   $E_0$ ait r\'eduction de type multiplicatif en $\LL$ et   que $E_0$ soit semi-stable. On  peut alors conclure comme dans l'alin\'ea pr\'ec\'edent. 
  
Supposons donc d\'esormais que l'on a 
 \begin{equation}
 \label{(18.3)}
 v_{\LL}(abc)=1.
\end{equation}

  \begin{lemme}   \label{L:lemme14} Quitte \`a multiplier $(a,b,c)$ par une unit\'e convenable de $O_{K_2}$, on a 
$$v_{\LL}(N_{E_0})\in \lbrace 5,6,8\rbrace.$$
 \end{lemme} 
 
 \begin{proof} D'apr\`es \eqref{(18.2)} et \eqref{(18.3)}, on a 
 \begin{equation}
  \label{(18.4)}
 v_{\LL}\left(c_4(E_0)\right)=16,\quad  v_{\LL}\left(c_6(E_0)\right)=24,\quad  v_{\LL}\left(\Delta(E_0)\right)=42.
  \end{equation}
On a $v_{\LL}\left(j_{E_0}\right)=6$, donc $E_0$ a potentiellement bonne r\'eduction en $\LL$. 

On peut supposer que l'on a $v_{\LL}(b)=1$, auquel cas $v_{\LL}(ac)=0$. 
 Par ailleurs, il existe une unit\'e $\varepsilon$ de $O_{K_2}$ telle que l'on ait $\varepsilon a+1 \equiv 0 \pmod 4$ (Appendice, th.~\ref{T:th17} et lemme~\ref{L:lemme16}). En  rempla\c cant  $(a,b,c)$ par 
 $(\varepsilon a, \varepsilon b, \varepsilon c)$, on se ram\`ene au cas o\`u l'on a
  \begin{equation}
   \label{(18.5)}
a^{13}+1\equiv 0 \pmod 4.
   \end{equation}
  Le mod\`ele \eqref{(18.1)} n'est pas minimal en $\LL$. En effet, posons 
  $$x=\alpha^4X\quad \hbox{et}\quad y=\alpha^6Y+\alpha^4X.$$
 On obtient comme nouveau mod\`ele
$$(W) : Y^2+\frac{2}{\alpha^2}XY=X^3+\left(\frac{b^{13}-a^{13}-1}{\alpha^4}\right)X^2-\frac{(ab)^{13}}{\alpha^8} X.$$
D'apr\`es \eqref{(18.5)} et le fait que $2$ soit  associ\'e \`a $\alpha^4$, c'est un mod\`ele entier. On a
$$c_4(E_0)=\alpha^8 c_4(W),\quad  c_6(E_0)=\alpha^{12} c_6(W),\quad \Delta(E_0)=\alpha^{24}\Delta(W),$$
d'o\`u (formules \eqref{(18.4)})
$$v_{\LL}(c_4(W))=8, \quad v_{\LL}\left(c_6(W)\right)=12,\quad v_{\LL}(\Delta(W))=18.$$
On v\'erifie avec les tables de \cite{Papadopoulos} que le type de N\' eron en $\LL$ de $(W)$ est  I$_6^*$ ou bien que $(W)$ n'est pas minimal en $\LL$. Si le type de N\'eron est I$_6^*$, on a  $v_{\LL}(N_{E_0})=8$. Si $(W)$ n'est pas minimal, le triplet de valuations de ses invariants minimaux  en $\LL$ est   
 $(4,6,6)$. On constate alors que son  type de N\'eron est II, auquel cas 
$v_{\LL}(N_{E_0})=6$, ou bien que son   type de N\'eron est III et  on a  $v_{\LL}(N_{E_0})=5$, d'o\`u le lemme.
  \end{proof} 
 
 Supposons $(a,b,c)$ normalis\'e comme dans l'\'enonc\'e du  lemme pr\'ec\'edent. 
 Notons $\calS_2^+(\LL^r)$ le $\C$-espace vectoriel engendr\'e par les newforms modulaires paraboliques de Hilbert sur $K_2$, de poids parall\`ele $2$ et de niveau $\LL^r$. 
On a (\cite{lmfdb})
 $$\dim \calS_2^+(\LL^5)=1,\quad \dim \calS_2^+(\LL^6)=3\quad \hbox{et}\quad \dim \calS_2^+(\LL^8)=8.$$
Compte tenu des lemmes \ref{L:lemme11} et \ref{L:lemme13}, les conditions du th\'eor\`eme d'abaissement du niveau sont satisfaites (\cite{FS1}, th. 7). Il existe donc   $\ff\in \calS_2^+(\LL^r)$ avec $r\in \lbrace 5,6,8\rbrace$ et un id\'eal premier $\pp$ au-dessus de $13$ dans $O_{\Q_{\ff}}$,  tels que
$$\rho_{\ff,\pp}\simeq \rho_{E_0,13}.$$
Consid\'erons alors le nombre premier $79$. Il est totalement d\'ecompos\'e dans $K_2$. Soit $\fq$ un id\'eal premier de $O_{K_2}$ au-dessus de $79$. 
On constate que l'on a 
$$a_{\fq}(\ff)\in \ \lbrace -8,0,8 \rbrace.$$
Si $E_0$ a r\'eduction multiplicative en $\fq$, on a 
$$a_{\fq}(\ff)\equiv \pm 2 \pmod \pp,$$ ce qui n'est pas. Ainsi $E_0$ a bonne r\'eduction en $\fq$. On a donc la congruence
$$a_{\fq}(E_0)\equiv a_{\fq}(\ff) \pmod \pp.$$
On a $v_{\fq}(abc)=0$, donc $a^{13}, b^{13}, c^{13}$ sont des racines $6$-i\`emes de l'unit\' e modulo $\fq$. On  a ainsi
$$a^{13}, b^{13}, c^{13}\equiv 1, 23, 24, 55, 56, 78 \pmod \fq.$$
L'\'egalit\'e $a^{13}+b^{13}+c^{13}=0$ implique
$$(a^{13}, b^{13})\equiv (1,23), (1,55),(23,1),(23,55),(24,56),(24,78),(55,1),$$
$$(55,23),(56,24),(56,78), (78,24), (78,56)\pmod \fq.$$
Dans tous les cas, on obtient 
$$a_{\fq}(E_0)=\pm 4,$$ d'o\`u la contradiction cherch\'ee et le r\'esultat pour $p=13$.

\subsection{L'assertion 2}

 \begin{lemme}   \label{L:lemme15} Pour tout $p>6724$, la  repr\'esentation $\rho_{E_0,p}$ est irr\'eductible.
\end{lemme} 
 
\begin{proof} On a $D_{K_3}=2^{31}$ et $K_3=\Q(\alpha)$ o\`u $$\alpha^8- 8\alpha^6 + 20\alpha^4 - 16\alpha^2 + 2=0.$$  Posons 
$$u_1=(-\alpha^6 - 2\alpha^5 + 5\alpha^4 + 10\alpha^3 - 4\alpha^2 - 9\alpha - 1)^2,\quad
u_2=(\alpha^7 + \alpha^6 - 6\alpha^5 - 5\alpha^4 + 9\alpha^3 + 5\alpha^2 - 3\alpha - 1)^2,$$
$$u_3=(-2\alpha^7 - 2\alpha^6 + 11\alpha^5 + 10\alpha^4 - 13\alpha^3 - 9\alpha^2 + \alpha + 1)^2.$$
Ce sont des unit\'e totalement positives de $O_{K_3}$. En utilisant le th\'eor\`eme \ref{T:th13} avec les $u_i$, 
on  v\'erifie que   on a l'implication
$$R_{K_3}\equiv 0 \pmod p  \ \Rightarrow \ p\leq 607.$$
Par ailleurs, on a 
$p_0<6724$ (\cite{Oesterle}), d'o\`u l'assertion.
 \end{proof} 
 
 Les conditions du th\'eor\`eme \ref{T:th14}  sont satisfaites.
 On constate avec {\tt Magma} que  l'on a $|\calH|=40$. \`A conjugaison pr\`es,  $\calH$ est form\'e de quatre newforms $\ff$  telles que $[\Q_{\ff}:\Q]=4$ et d'une newform dont le corps de rationalit\'e est de degr\'e 24 sur $\Q$. 
 
 Les nombres premiers 31 et 97 sont totalement d\' ecompos\'es dans $K_3$. En utilisant la proposition 5, et  en prenant pour $\fq$ un id\'eal premier de $O_{K_3}$ au-dessus de $31$,  puis  un id\' eal premier au-dessus de $97$, on obtient alors le r\'esultat.
\bigskip

{\large{{\bf{Partie 5. Appendice}}}}

 Soit $K$ un corps de nombres totalement r\'eel.  
 Notons  $K^{4O_K}$ le corps de classes de rayon modulo $4O_K$ sur $K$.
On \'etablit ici  l'\'enonc\'e  suivant, que l'on utilise, tout au moins sa premi\`ere assertion, dans les d\'emonstrations du th\' eor\`eme~\ref{T:th5}, de la proposition~\ref{P:prop3}  et du lemme~\ref{L:lemme14}. 

\bigskip

\begin{theoreme} \label{T:th17}

1) Supposons $2$ totalement ramifi\'e dans $K$ et $h_K^+=1$. Alors,  on a $K=K^{4O_K}$.

2) Supposons  $K=K^{4O_K}$. Alors, on a $h_K^+=1$. 
 \end{theoreme}
 
 En particulier :
 
\begin{corollaire} \label{C:cor1} Supposons $2$ totalement ramifi\'e dans $K$. On a $h_K^+=1$ si et seulement si $K=K^{4O_K}$.
\end{corollaire}

\subsection{R\'esultats pr\'eliminaires} Notons $U_K$ le groupe des unit\'es de $O_K$ et  $h_K$ le nombre de classes de $K$. Rappelons que $d$ d\'esigne le degr\'e de $K$ sur $\Q$. Posons  
$$G=(O_K/4O_K)^*.$$
Soit  $\varphi : U_K\to G$ le morphisme qui \`a $u\in U_K$ associe $u+4O_K$.

\begin{lemme} \label{L:lemme16} On a $K=K^{4O_K}$ si et seulement si $h_K=1$ et $\varphi$ est une surjection sur $G$. 
\end{lemme}

\begin{proof} C'est cons\'equence directe de proposition 3.2.3 de \cite{Cohen}.
\end{proof}

\begin{lemme} \label{L:lemme17} Les deux conditions suivantes sont \'equivalentes :

1)   On a $h_K^+=1$.

2)   On a $h_K=1$ et toute unit\'e totalement positive est un carr\'e dans $K$. 
\end{lemme}

\begin{proof} On a l'\'egalit\'e
(cf. \cite{Cohen}, prop. 3.2.3,  avec pour  $\mm$ le produit des places \`a l'infini)
$$h_K^+=\frac{2^d}{[U_K:U_K^+]}h_K.$$
Supposons $h_K^+=1$. On a alors $h_K=1$, puis  $[U_K:U_K^+]=2^d$. D'apr\`es le th\'eor\`eme de Dirichlet, on a $[U_K:U_K^2]=2^d$, d'o\`u $U_K^+=U_K^2$.
Inversement, si  $h_K=1$ et $U_K^+=U_K^2$, la formule ci-dessus implique $h_K^+=1$.
\end{proof}

\begin{lemme} \label{L:lemme18} Supposons $2$ totalement ramifi\'e dans $K$. On a
$$|G/G^2|=2^d.$$
\end{lemme} 

\begin{proof} Soit $a$ un \' el\'ement de $G$.  Soit $\LL$  l'id\'eal premier de $O_K$ au-dessus de $2$. On a $O_K/\LL=\F_2$, donc il existe $x\in \LL$ tel que $a=1+x +4O_K$.  On  a $a^2=1+x(2+x)+4O_K$, d'o\`u il  r\'esulte que l'on a 
$$a^2=1 \Leftrightarrow x\in 2O_K.$$
Posons  $G[2]=\lbrace z\in G \ | \ z^2=1\rbrace$.   On en d\'eduit une application $f : G[2]\to  2O_K/4O_K$ d\'efinie pour  tout 
$a\in G[2]$ par l'\'egalit\'e 
$$f(a)=x+4O_K\quad \hbox{o\`u}\quad a=1+x+4O_K.$$
C'est un isomorphisme de groupes. Le groupe  $2O_K/4O_K$ est   isomorphe  \`a $O_K/2O_K$ qui est d'ordre $2^d$. Par ailleurs, les ordres  de  $G[2]$ 
et  $G/G^2$ sont \'egaux,   d'o\`u le lemme.
\end{proof}

\begin{lemme} \label{L:lemme19} Supposons $2$ totalement ramifi\'e dans $K$. Les deux conditions suivantes sont \'equivalentes :

1)  On a  $K=K^{4O_K}$.

2)   On a $h_K=1$ et toute unit\'e congrue \`a un carr\'e modulo $4$  est un carr\'e dans $K$. 
\end{lemme}

\begin{proof} Notons $\psi : U_K\to G\to G/G^2$ le morphisme naturel d\'eduit de $\varphi$.

Supposons $K=K^{4O_K}$.  Le morphisme $\psi$ est une surjection
 (lemme~\ref{L:lemme16}). Les ordres  de $U_K/U_K^2$ et $G/G^2$ sont \'egaux (lemme~\ref{L:lemme18}). Par suite, $U_K^2$ est le noyau de $\psi$, donc 
toute unit\'e congrue \`a un carr\'e modulo $4$  est un carr\'e.

Inversement, supposons la seconde condition   satisfaite. D\'emontrons que le morphisme $\varphi$ est surjectif, ce qui,  d'apr\`es le lemme~\ref{L:lemme16},  impliquera  la premi\`ere assertion.  D'apr\`es l'hypoth\`ese faite, le noyau de $\psi$ est $U_K^2$. Les ordres de $U_K/U_K^2$ et $G/G^2$ \'etant \'egaux, on en d\'eduit  que $\psi$ est une surjection. Ainsi, l'image de $\varphi(U_K)$ dans $G/G^2$ est $G/G^2$.  
 Parce que $2$ est totalement ramifi\'e dans $K$, $G$ est un $2$-groupe (\cite{Cohen}, p. 137). Il en r\'esulte   que $\varphi(U_K)=G$, d'o\`u le r\'esultat. 
(Si $p$ est premier, le sous-groupe de Frattini   d'un $p$-groupe ab\'elien fini $A$ est $A^p$, voir  par exemple \cite{Rotman}, th. 5.48. Si $B$ est un sous-groupe de $A$ tel que $BA^p=A$, on a donc $A=B$.)  
\end{proof}

\subsection{Fin de la d\' emonstration du th\'eor\`eme 17} 

1) Supposons $2$ totalement ramifi\'e dans $K$ et $h_K^+=1$. Soit $u$ une unit\'e de $O_K$ congrue \`a un carr\'e modulo $4$. L'extension $K\left(\sqrt{u}\right)/K$ est alors partout non ramifi\'ee aux places finies de $K$ (\cite{Cox}, lemme 5.32, p. 114). On a $h_K^+=1$, donc $u$ est un carr\'e dans $K$. Vu que $h_K=1$, on a donc $K=K^{4O_K}$  (lemme~\ref{L:lemme19}). 

2) Supposons  $K=K^{4O_K}$. On a $h_K=1$. Soit $u$ une unit\'e totalement positive de $O_K$. D'apr\`es le lemme~\ref{L:lemme17}, il s'agit de montrer que $u$ est un carr\'e dans $K$.  L'extension $K\left(\sqrt{u}\right)/K$  est non ramifi\'ee aux places \`a l'infini  et en dehors des id\'eaux premiers de $O_K$ au-dessus de $2$. 
Le conducteur de  l'extension $K\left(\sqrt{u}\right)/K$ est  \'egal \`a son  discriminant (\cite{Cassels}, p. 160), qui divise $4O_K$. Par suite,  $K\left(\sqrt{u}\right)$ est contenu dans $K^{4O_K}$, d'o\`u $K\left(\sqrt{u}\right)=K$ et notre assertion.

Alain Kraus


\bigskip


\begin{thebibliography}{}




\bibitem{AnniSiksek} S.  Anni et S. Siksek,
{\em On the generalized Fermat equation $x^{2\ell} + y^{2m}=z^p$ for $3 \leq p \leq 13$ },
Algebra  Number Theory {\bf 10} (2016), 1147--1172.

 \bibitem{Pari} C. Batut, D. Bernardi, K. Belabas, H. Cohen et M. Olivier, {\tt Pari-gp}, version 2.7.3,  
Universit\'e  de Bordeaux I, (2016).

\bibitem{MAGMA} W.\ Bosma, J.\ Cannon et C.\ Playoust: {\em The Magma
Algebra System I: The User Language}, J.\ Symb.\ Comp.\ {\bf 24} (1997),
235--265. (voir aussi {\tt http://magma.maths.usyd.edu.au/magma/})  

\bibitem{BCDT} C. Breuil, B. Conrad, F. Diamond et R. Taylor, 
{\em On the modularity of elliptic curves over $\Q$ : wild $3$-adic exercices}, 
J. Amer. Math. Soc. {\bf{14}} (2001), 843--939.

\bibitem{Browkin} J. Browkin,
{\em The $abc$-conjecture for Algebraic Numbers},
Acta Math. Sinica, English Series {\bf{22}} (2006), 211--222.

\bibitem{BruiniNajman}  P. Bruin et F. Najman, 
{\em A criterion to rule out torsion groups for elliptic curves over number fields} Research in Number Theory
{\bf 2} (2016), 1--13.



\bibitem{Cassels} J. W. S. Cassels et A. Fr\"olich, 
{\em Algebraic Number Theory}, Proceedings London Mathematical Society, 
Academic Press (1967). 

\bibitem{Cohen} H. Cohen, 
{\em Advanced Topics in Computational Number Theory}, Graduate Texts in Mathematics {\bf 193}, Springer, 2000.

\bibitem{Cox} D. A. Cox,
{\em Primes of the form $x^2+ny^2$},
Wiley Interscience (1989).


\bibitem{Cremona} J. E. Cremona, 
{\em Algorithms for modular elliptic curves},
Cambdridge University Press, Second edition (1997).

\bibitem{Dembelecremona} J. E. Cremona et L. Demb\'el\'e,
{\em Modular forms over number fields}, Expository notes.


\bibitem{Dahmen} S. R. Dahmen,
{\em Classical  and modular methods applied to Diophantine equations},
PhD Thesis, Utrecht University (2008).






\bibitem{Dembelevoight} L. Demb\'el\'e et J. Voight,
{\em Explicit methods for Hilbert modular forms},
Hilbert modular forms and Galois deformations, 
Adv. Courses Math. CRM Barcelona, Birkha\"user/Springer, Basel (2013), 135--198.


\bibitem{Derickx} M. Derickx,  
{\em Torsion points on elliptic curves over number fields of small degree},
Doctoral Thesis,  \url {http://hdl.handle.net/1887/43186}, Leiden University Repository (2016).



\bibitem{FS1} N. Freitas et S. Siksek, 
{\em The asymptotic Fermat's Last Theorem for five-sixths of real quadratic fields} Composition Math. 
{\bf 151} (2015), 1395--1415.


\bibitem{FLS} N. Freitas, Bao V. Le Hung et S. Siksek, 
{\em Elliptic curves over real quadratic fields are modular}, 
Invent. Math. {\bf 201} (2015), 159--206. 





\bibitem{FS3} N.\ Freitas et S.\ Siksek,
{\em Fermat's Last Theorem over some small real quadratic fields},
Algebra  Number Theory {\bf 9} (2015), 875--895.

\bibitem{FK} T. Fukuda et K. Komatsu,
{\em Weber's class number problem in the cyclotomic $\Z_2$-extension of $\Q$, III},
Int. J. Number Theory\ {\bf 7} (2011), 1627--1635.


\bibitem{GR} B. H. Gross et D. E. Rohrlich,
{\em Some results on the Mordell-Weil group of the Jacobian of the Fermat curve},
Invent. Math. {\bf{44}} (1978), 201-224.


\bibitem{JarvisMeekin} F. Jarvis et P. Meekin, 
{\em The Fermat equation over $\Q\bigl(\sqrt{2}\bigr)$},
J. Number Theory\ {\bf 109} (2004), 182--196.
 
 \bibitem{KlassenTzermias} M. Klassen et P. Tzermias,
{\em Algebraic points of low degree on the Fermat quintic},
  Acta Arith. {\bf{82}} (1997), 393-401.
      
\bibitem{Kraus2007} A.\ Kraus,
{\em Courbes elliptiques semi-stables sur les corps de nombres},
Int. J. Number Theory\ {\bf 3} (2007), 611--633.


\bibitem{Kraus2017} A.\ Kraus,
{\em Quartic points on the Fermat quintic},
ArXiv : 1706.03569v1 (2017).




\bibitem{lmfdb} The {LMFDB Collaboration},
{\em The L-functions and Modular Forms Database},
\url{http://www.lmfdb.org}, 2013, [Online; accessed 16 July 2016].





\bibitem{Merel96} L. Merel, 
{\em Borne pour la torsion des courbes elliptiques sur les corps de nombres},
Invent. Math.  {\bf 124} (1996),  437--449.


\bibitem{Oesterle} J. Oesterl\'e,
{note non publi\'ee}, (1996).

\bibitem{Papadopoulos} I.\ Papadopoulos, 
{\em Sur la classification de {N}\'eron des courbes elliptiques en caract\'eristique r\'esiduelle {$2$} et {$3$}},
J.\ Number Theory, {\bf 44} (1993),   119--152.

\bibitem{Parent} P. Parent,
{\em No $17$-torsion on elliptic curves over cubic fields}
J. Th\'eorie des Nombres de Bordeaux, {\bf{15}} (2003), 831-838.


\bibitem{Ri} K. Ribet, 
{\em On modular representations of $\Gal(\overline \Q/\Q)$ arising from modular forms},
Invent. Math, {\bf 100} (1990), 431--476.

\bibitem{Rotman} J. J. Rotman,
  {\em An Introduction to the Theory of Groups}, Fourth Edition, Graduate Texts in Mathematics {\bf{148}}, Springer, 1995.
  
\bibitem{Sengunsiksek} M. H. Seng\"un et S. Siksek, 
{\em On the asymptotic Fermat's Last Theorem over number fields},
\`a para\^itre dans la revue Commentarii Mathematici Helvetici, ArXiv : 1609.04458v3 (2016).
  
\bibitem{Se1} J.-P.\ Serre,
{\em Propri\'et\'es galoisiennes des points d'ordre fini des 
courbes elliptiques},
Invent.\ Math.\ {\bf 15} (1972) 259--331.

\bibitem{Se2} J.-P.\ Serre,
{\em Sur les repr\'esentations modulaires de degr\'e 2 de $\Gal(\overline \Q/\Q)$},
Duke  Math. J.\ {\bf 54} (1987) 179--230.

\bibitem{Siegel} C. L. Siegel,
{\em  \"Uber einige Anwendungen  diophantischer Approximationen},
Abhandlungen der  Preussischen  Akademie der  Wissenschaften, Walter de Gruyter, Berlin,  (1929), 1--41.
      
\bibitem{Silverman} J.\ H. Silverman,
{\em The Arithmetic of Elliptic Curves},
Second Edition, Graduate Texts in Mathematics {\bf 106}, Springer,  2009.

\bibitem{Smart1}  N. P. Smart,
 {\em  The algorithmic resolution of Diophantine equations},
 London Mathematical Society, Student Texts, {\bf{41}} Cambridge University Press, 1998.
 
 
  \bibitem{Taylorwiles} R. Taylor et  A. Wiles,
  {\em Ring-theoretic properties of certain Hecke algebras},
  Annals of Math.\ {\bf 141} (1995), 553--572.
 
  \bibitem{Thorne} J. Thorne,
 {\em Elliptic curves over $\Q_{\infty}$ are modular},
 Journal of the European Mathematical Society {\tt https://doi.org/10.17863/CAM.8465}, voir aussi ArXiv: 1505.04769v1 (2015).
 
\bibitem{Tzermias} P. Tzermias,
{\em Algebraic points of low degree on the Fermat curve of degree seven},
Manuscripta Math. {\bf{97}} (1998) 483--488.
  
  
  
\bibitem{Voight} J. Voight, Tables of totally real number fields, \url {https://math.dartmouth.edu/~jvoight/nf-tables/}.
  
  \bibitem{Waterhouse} W. Waterhouse, 
  {\em Abelian varieties over finite fields}, Ann. Sci. \'Ecole Norm. Sup. {\bf{4}} (1969), 521--560.
  
  
\bibitem{Wiles} A.\ Wiles,
{\em Modular elliptic curves and Fermat's Last Theorem},
Annals of Math.\ {\bf 141} (1995), 443--551.


\end{thebibliography}
\end{document}